\documentclass{tac}

\usepackage{tikz-cd}
\usepackage{amsmath, amssymb}
\usepackage{enumitem}
\usepackage[UKenglish]{babel}

\usepackage[colorlinks=true]{hyperref}
\hypersetup{allcolors=[rgb]{0.1,0.1,0.4}, breaklinks}

\author{Vincenzo Marra and Luca Reggio}

\thanks{The second author acknowledges financial support from Sorbonne Paris Cit\'e (PhD agreement USPC IDEX -- REGGI15RDXMTSPC1GEHRKE), and from the European
    Research Council (ERC) under the European Union's Horizon 2020
    research and innovation programme (grant agreement No.670624).}

\address{Dipartimento di Matematica ``Federigo Enriques'', Universit\`a degli Studi di Milano\\
 via Cesare Saldini 50, 20133 Milano, Italy\\[5pt]
 Department of Computer Science, University of Oxford\\
15 Parks Road, Oxford OX1 3QD, UK\\
}

\title{A characterisation of the category of compact Hausdorff spaces}

\copyrightyear{2020}



\eaddress{vincenzo.marra@unimi.it\CR luca.reggio@cs.ox.ac.uk}

\newtheorem{claim}{Claim} 

\renewcommand{\P}{\raisebox{.17\baselineskip}{\Large\ensuremath{\wp}}} % Power-set functor
\newcommand{\op}{\mathrm{op}} % opposite category
\DeclareMathOperator{\Vs}{\mathbb{V}} % Set of points of an object
\DeclareMathOperator{\Cg}{\mathbb{I}} % Smallest object containing a set of points
\DeclareMathOperator{\F}{\mathcal{F}} % Filters
\DeclareMathOperator{\B}{\mathcal{B}} % Boolean center (of space)
\DeclareMathOperator{\Ct}{\mathcal{C}} % Boolean center (of lattice)
\DeclareMathOperator{\pt}{pt} % Points
\DeclareMathOperator{\Sub}{Sub} % Subobjects
\DeclareMathOperator{\Spec}{Spec} % Spectrum
\DeclareMathOperator{\Hom}{Hom} % Hom
\DeclareMathOperator{\cl}{\mathfrak{c}} % closure operator
\DeclareMathOperator{\K}{\ensuremath{\mathcal{K}}} % lattice of closed subsets

\newcommand{\mono}{\hookrightarrow} % Monomorphism
\newcommand{\epi}{\twoheadrightarrow} % Regular epimorphism
\DeclareMathOperator{\0}{\mathbf{0}} % Initial object
\DeclareMathOperator{\1}{\mathbf{1}} % Terminal object 
\DeclareMathOperator{\2}{\mathbf{2}}
\renewcommand{\epsilon}{\varepsilon}
\renewcommand{\theta}{\vartheta}
\renewcommand{\phi}{\varphi}
\renewcommand{\restriction}{\mathord{\upharpoonright}}

\DeclareMathOperator{\Set}{\mathbf{Set}} % Sets
\DeclareMathOperator{\Setfin}{\mathbf{Set}_\mathit{f}} % Finite sets
\DeclareMathOperator{\KH}{\mathbf{KH}} % Compact Hausdorff spaces
\DeclareMathOperator{\BA}{\mathbf{BA}} % Boolean algebras
\DeclareMathOperator{\BStone}{\mathbf{Stone}} % Stone spaces
\DeclareMathOperator{\Top}{\mathbf{Top}} % Topological spaces
\DeclareMathOperator{\DL}{\mathbf{DL}} % Distributive lattices
\DeclareMathOperator{\SL}{\mathbf{SL}} % Meet-semilattices
\DeclareMathOperator{\X}{\mathbf{X}} 
\DeclareMathOperator{\C}{\mathbf{C}}
\DeclareMathOperator{\D}{\mathbf{D}}
\DeclareMathOperator{\dec}{Dec} % subcategory of decidable objects
\DeclareMathOperator{\prodec}{proDec} % subcategory of pro-decidable objects

\begin{document}

\maketitle
\begin{abstract}
We provide a characterisation of the category $\KH$ of compact Hausdorff spaces and continuous maps by means of categorical properties only. 	
To this aim we introduce a notion of \emph{filtrality} for coherent categories, relating certain lattices of subobjects to their Boolean centers.
Our main result reads as follows: \emph{Up to equivalence, $\KH$ is the unique non-trivial well-pointed pretopos which is filtral and admits all set-indexed copowers of its terminal object}.
\end{abstract}

%%%%%%%%%%%%%%%%%%%%%%%%%%%%%%%%%%
\section{Introduction}\label{s:introduction}
Several characterisations of the class of compact Hausdorff spaces are available in the literature. For example, de Groot's Theorem asserts that compact Hausdorff spaces form the only non-trivial, productive and closed-hereditary class of topological spaces which are absolutely closed and preserved under closed images. See~\cite[p.~51]{Wattel1968}, and~\cite{FT1970} for a categorical translation of this result. The category $\KH$ of compact Hausdorff spaces and continuous maps has been widely investigated in categorical topology. There, the characterisation of subcategories of the category of topological spaces is an important concern. In this direction, Herrlich and Strecker showed that $\KH$ is the unique non-trivial full epireflective subcategory of Hausdorff spaces that is \emph{varietal} in the sense of~\cite{Linton66}. See~\cite{HS1971}, and also~\cite{Richter1996}. 

A common feature of these characterisations is that they are all relative to an ambient class of topological spaces. 
To the best of our knowledge, the only abstract characterisations of the category of compact Hausdorff spaces were provided in~\cite{Richter1991, Richter1992}. 
For a discussion of these results, please see the end of this introduction.

Our main result, Theorem~\ref{t:main-characterisation-bis}, offers a new characterisation of the category of compact Hausdorff spaces. 
We prove that, up to equivalence, $\KH$ is the unique non-trivial well-pointed (and well-powered) pretopos which admits all set-indexed copowers of its terminal object, and satisfies a condition that we call \emph{filtrality}. This latter notion makes sense in any coherent category, and appears to be new with our paper. Filtrality asserts that every object is covered by one whose lattice of subobjects is isomorphic (in a canonical way) to the lattice of filters of its Boolean center. It may be understood as stating the existence of enough objects satisfying a form of compactness, Hausdorffness and zero-dimensionality.  The second prominent ingredient of our characterisation is the pretopos structure. A \emph{pretopos} is an exact and extensive category. While exactness is a distinguishing property of categories of algebras, extensivity abstracts a property typical of categories of spaces. In this sense, Theorem~\ref{t:main-characterisation-bis} hinges on the fact that $\KH$ has both a spatial  and an algebraic nature. We briefly comment on these two aspects.

The evident spatial nature of $\KH$ has proved fruitful from a duality-theoretic viewpoint.
Several dualities for compact Hausdorff spaces were discovered in the first half of the last century. See, e.g.,~\cite{Kakutani41, KreinKrein40, Yosida41}. The best-known result probably is \emph{Gelfand-Naimark duality} between compact Hausdorff spaces and commutative unital $\mathrm{C}^*$-algebras~\cite{GN1943}.
The concept of norm, central in the definition of $\mathrm{C}^*$-algebra, is not algebraic in nature. However, Duskin showed in 1969 that the dual of the category of compact Hausdorff spaces is monadic over the category $\Set$ of sets and functions~\cite[{}5.15.3]{Duskin1969}. In fact, $\KH^{\op}$ is equivalent to a variety of algebras. Although operations of finite arity do not suffice to describe any such variety, Isbell proved that finitely many finitary operations, along with a single operation of countably infinite arity, are sufficient~\cite{Isbell1982}. In~\cite{MR2017} we provided a finite axiomatisation of a variety of algebras dually equivalent to $\KH$. For more on the axiomatisability of the dual of compact Hausdorff spaces, we refer the reader to~\cite{Banaschewski1984, Rosicky1989}.

On the other hand, the category $\KH$ itself also has an algebraic nature. This was first pointed out by Linton, who proved that the category of compact Hausdorff spaces is varietal, hence monadic over $\Set$~\cite[Section 5]{Linton66}. An explicit description of the corresponding equational theory was later given by Manes, see~\cite[Section 1.5]{Manes1976} for a detailed exposition, who showed that compact Hausdorff spaces are precisely the algebras for the ultrafilter monad on $\Set$.
This algebraic nature appears to be one of the distinctive features of the category of compact Hausdorff spaces among the categories of topological spaces, and was exploited by Herrlich and Strecker, and by Richter, to obtain the aforementioned characterisations of $\KH$.

Our characterisation of the category of compact Hausdorff spaces can be compared to Lawvere's \emph{Elementary Theory of the Category of Sets} (ETCS), outlined in~\cite{Lawvere1964}. For a more detailed exposition, see~\cite{Lawvere2005}.
Lawvere gives eight elementary axioms (in the language of categories) such that every complete and cocomplete category satisfying these axioms is equivalent to $\Set$. Some of his axioms appear \emph{verbatim} in our characterisation, e.g., the existence of enough points (\emph{elements}, in Lawvere's terminology). 
Where Lawvere's characterisation and ours diverge is about the existence of infinite ``discrete'' objects. While the second and third axioms of ETCS jointly imply the existence of a natural numbers object satisfying Primitive Recursion~\cite[Theorem 1]{Lawvere2005}, we prescribe filtrality which forbids the existence of infinite discrete objects.\footnote{Note that $\KH$ does admit a natural numbers object, namely the Stone-\v{C}ech compactification $\beta(\mathbb{N})$. However, this natural numbers object is not stable under taking slice categories, thus preventing Primitive Recursion from holding in $\KH$. For more details, cf.~\cite[Section~A2.5]{Elephant1}.} In a sense, when compared to Lawvere's, our characterisation clarifies to what extent the categories $\Set$ and $\KH$ are similar, and where they differ. 
Let us mention that Schlomiuk adapted Lawvere's ETCS to characterise the category of topological spaces~\cite{Schlomiuk1970}. However, her result does not bear a greater resemblance to ours than Lawvere's.

\medskip
The paper is structured as follows. In Section~\ref{s:coherent} we provide some background on coherent categories, which capture a part of the structure of pretopos. It turns out that this structure suffices for a large part of the construction leading to the main result. 
In Section~\ref{s:top-representation}, we study the functor assigning to each object of a well-pointed coherent category $\X$ its set of points, {\it alias} global elements. We focus on the situation where this functor admits a lifting to the category of topological spaces, yielding a topological representation of $\X$. 
The notion of filtrality is introduced in Section~\ref{s:filtrality}. We use it to prove that the topological representation of $\X$ lands in the category of compact Hausdorff spaces.
The full pretopos structure on $\X$ is considered in Section~\ref{s:main} to prove our main result, Theorem~\ref{t:main-characterisation-bis}, providing a characterisation of $\KH$. 
The last two sections of the paper are offered by way of an addendum. In Section~\ref{s:decidable}, we characterise the category of \emph{Stone spaces}---consisting of zero-dimensional compact Hausdorff spaces and continuous maps---in the spirit of Theorem~\ref{t:main-characterisation-bis}. Finally, in Section~\ref{s:epilogue}, we exploit our characterisation of $\KH$ to give a proof of the folklore result stating that the exact (equivalently, pretopos) completion of the category of Stone spaces is $\KH$.

\medskip
We end this introduction with a discussion of Richter's results in~\cite{Richter1991, Richter1992} that are relevant to our paper. Of the characterisations of $\KH$ in~\cite[Corollary 4.7]{Richter1991} and~\cite[Remark 4.7]{Richter1992}, we shall only consider the latter, which we regard as an improvement of the former.
In~\cite{Richter1992}, the author derives Remark~4.7 from Theorem~2.2, a more general result characterising the full subcategories of $\KH$ which contain all the Stone-\v{C}ech compactifications of discrete spaces. Unlike Richter's, our main result is based on the notion of coherent category. Indeed, as noted above, several of our intermediate results and constructions make sense in that context. We expect this fact to play a role in future research addressing, for instance,  categories of non-Hausdorff or ordered topological spaces. More generally, a number of Richter's axioms are, \textit{prima facie}, quite different from ours. To provide some mathematical substance to these heuristic comments, in Theorem~\ref{t:Richter} we offer a proof of Richter's result from our Theorem~\ref{t:main-characterisation-bis}. The main point is to show how Richter's axioms entail the pretopos structure. Since Richter assumes his category to have effective equivalence relations, the crux of the matter is to deduce regularity and extensivity from Richter's axioms. For details, please see Section~\ref{s:Richter-thm}.

\begin{notation}
Assuming they exist, the initial and terminal objects of a category are denoted  $\0$ and $\1$, respectively. The unique morphism from an object $X$ to $\1$ is $!\colon X\to\1$. The coproduct of two objects $X_1,X_2$ is written $X_1+X_2$. For arbitrary (set-indexed) coproducts we use the notation $\sum_{i\in I}{X_i}$. A monomorphism (resp.\ regular epimorphism) from an object $X$ to an object $Y$ is denoted  $X\mono Y$ (resp.\ $X\epi Y$). If $S$ is a set, $\P{(S)}$ is its power-set.
\end{notation}

%%%%%%%%%%%%%%%%%%%%%%%%%%%%%%%%
\section{Coherent categories}\label{s:coherent}
We recall some basic facts about coherent categories that will be used in the remainder of the paper. For a more thorough treatment, the reader can consult~\cite[Sections A1.3, A1.4]{Elephant1} or~\cite[Chapter 3]{MR77}. 

Given an object $X$ of a category $\C$, the collection of all monomorphisms with codomain $X$ admits a pre-order $\leq$ defined as follows. For any two monomorphisms $m_1\colon S_1\mono X$ and $m_2\colon S_2\mono X$, set $m_1\leq m_2$ if, and only if, there exists a morphism $S_1\to S_2$ in $\C$ such that the following diagram commutes.
\[\begin{tikzcd}
S_1 \arrow[dashed]{d} \arrow[hookrightarrow]{r}{m_1} & X \\
S_2 \arrow[hookrightarrow]{ur}[swap]{m_2} &
\end{tikzcd}\]
We can canonically associate with this pre-order an equivalence relation $\sim$, by setting $m_1\sim m_2$ if, and only if, $m_1\leq m_2$ and $m_2\leq m_1$. Note that $m_1\sim m_2$ precisely when there is an isomorphism $f\colon S_1\to S_2$ satisfying $m_1=f\circ m_2$. A ${\sim}$-equivalence class of monomorphisms with codomain $X$ is called a \emph{subobject} of $X$, and the collection of all subobjects of $X$ is denoted by $\Sub{X}$. The pre-order $\leq$ induces a partial order on $\Sub{X}$, that we denote again by $\leq$. When no confusion arises, we abuse notation and denote a subobject of $X$ by the domain of one of its representatives.
\begin{assumption}
A priori, an object can have a proper class of subobjects, as opposed to a set. For simplicity, throughout this paper we shall assume that all categories under consideration are \emph{well-powered}, i.e.\ $\Sub{X}$ is a set for every object $X$ in the category.
\end{assumption}
If the category $\C$ has pullbacks, then each poset of subobjects in $\C$ is a $\wedge$-semilattice. Just observe that, for any object $X$ of $\C$, the infimum in $\Sub{X}$ of two subobjects $m_1\colon S_1\mono X$ and $m_2\colon S_2\mono X$ is given by the pullback of $m_1$ along $m_2$ (recall that in any category the pullback of a mono along any morphism, if it exists, is a mono). The top element of $\Sub{X}$ is the identity $X\to X$. Moreover, for any morphism $f\colon X\to Y$ in $\C$, the \emph{pullback functor}
\[f^*\colon \Sub{Y}\to \Sub{X}\]
sending a subobject $m\colon S\mono Y$ to its pullback along $f$ is a $\wedge$-semilattice homomorphism preserving the top element. 
Thus, whenever $\C$ is a category with pullbacks, there is a well-defined functor
\begin{equation}\label{eq:semilattice-hyperdoctrine}
\Sub\colon \C^{\op}\to \SL
\end{equation}
into the category $\SL$ of $\wedge$-semilattices with top elements and semilattice homomorphisms preserving the top elements, which sends a morphism $f$ in $\C$ to $f^*$. Next, we look at the case where the pullback functors are upper (or right) adjoint.

The \emph{image} of a morphism $f\colon X\to Y$ in $\C$, if it exists, is the unique subobject $m\colon S\mono Y$ such that:
\begin{itemize}
\item $f$ factors through $m$;
\item for any subobject $m'\colon S'\mono Y$ through which $f$ factors, $m$ factors through $m'$.
\end{itemize}
That is, the image of $f$ is the smallest subobject of $Y$, in the partial order $\leq$ of $\Sub{Y}$, through which $f$ factors. Henceforth, we denote by $\exists_f(X)$ the image of $f$. In particular, the morphism $f$ factors as 
\begin{equation}\label{eq:factor-image}
X\to\exists_f(X)\mono Y.
\end{equation}
Moreover, there is an order-preserving function
\[
\exists_f\colon \Sub{X}\to \Sub{Y}
\]
sending a subobject $m\colon S\mono X$ to the image of the composition $f\circ m\colon S\to Y$. It is not difficult to see that this function is lower (or left) adjoint to the pullback functor $f^*$. In symbols,
\[
\exists_f\dashv f^*.
\]
We say that the image $\exists_f(X)$ is \emph{pullback-stable} if, for any morphism $g\colon Z\to Y$, taking the pullback of diagram~\eqref{eq:factor-image} along $g$ yields the image-factorisation of the pullback of $f$ along $g$.
\begin{definition}\label{d:regular-category}
A \emph{regular category} is a category with finite limits in which every morphism has a pullback-stable image. 
\end{definition}
\begin{example}
We give some examples of categories that are, or are not, regular.
\begin{itemize}
\item Any $\wedge$-semilattice with top element, regarded as a category, is regular.
\item Every variety of (Birkhoff) algebras is a regular category, with morphisms all the homomorphisms. Images are the homomorphic images. In particular, the category $\Set$ of sets and functions is regular.
\item The category $\KH$ of compact Hausdorff spaces and continuous maps, and its full subcategory $\BStone$ on the Stone spaces (i.e., zero-dimensional compact Hausdorff spaces),  are regular. Finite limits and images are liftings of those in $\Set$. In particular, images are simply continuous images and they are stable under pullbacks. 
\item The category $\Top$ of topological spaces and continuous maps is not regular because images are given by regular epis, which are not stable under pullbacks. See, e.g.,~\cite[p.\ 180]{PT2004}.
\end{itemize}
\end{example}
Every regular category admits a (regular epi, mono) factorisation system which is stable under pullbacks. This is given by taking the factorisation of a morphism through its image.
If, in addition to the requirements for a regular category, we ask that finite joins of subobjects exist and are preserved by the pullback functors, we arrive at the notion of coherent category.
\begin{definition}
A \emph{coherent category} is a regular category in which every poset of subobjects has finite joins and, for every morphism $f\colon X\to Y$, the pullback functor \[f^*\colon \Sub{Y}\to \Sub{X}\] preserves them. 
\end{definition}
\goodbreak
\begin{example}\label{ex:coherent-categories}
We give some examples of categories that are, or are not, coherent.
\begin{itemize}
\item Any bounded distributive lattice, regarded as a category, is coherent.
\item The categories $\KH$ and $\BStone$ are coherent. The join of two subspaces is simply their (set-theoretic) union, and is stable under pullbacks. 
\item $\Top$ is not regular and, a fortiori, not coherent.
\item Not every variety of (Birkhoff) algebras forms a coherent category. For instance, it will follow from Lemma~\ref{l:posets-coherent-distributive-lattice} below that the category of groups and group homomorphisms is not coherent, because the lattice of all subgroups of a given group is not distributive, in general. 
\end{itemize}
\end{example}
Every coherent category admits an initial object $\0$ which is \emph{strict}, i.e.\ every morphism $X\to \0$ is an isomorphism. This is the content of the next lemma.
For a proof see, e.g.,~\cite[A.1.4]{Elephant1}.
\begin{lemma}\label{l:cohe-strict-initial}
Every coherent category has a strict initial object $\0$. It can be defined as the least element of $\Sub{\1}$, where $\1$ is the terminal object.
\end{lemma}
By definition, in a coherent category $\C$ the posets of subobjects are bounded lattices and the pullback functors are lattice homomorphisms preserving the top elements. 
For every object $X$ of $\C$, the bottom element of $\Sub{X}$ is the unique morphism $\0\to X$. It follows at once from the previous lemma that the pullback functors preserve also bottom elements.
The lattice $\Sub{X}$ has a further important property, namely it is distributive. This follows essentially from the fact that the maps $S\wedge -\colon \Sub{X}\to\Sub{S}$, for any subobject $m\colon S\mono X$, coincide with the pullback functors $m^*$ and thus preserve finite suprema. We record this fact for future reference. For a detailed proof see, e.g.,~\cite[Lemma A.1.4.2]{Elephant1}.
\begin{lemma}\label{l:posets-coherent-distributive-lattice}
For every object $X$ of a coherent category $\C$, the lattice of subobjects $\Sub{X}$ is distributive.
\end{lemma}
Therefore, for every coherent category $\C$, the functor $\Sub\colon\C^{\op}\to\SL$ from~\eqref{eq:semilattice-hyperdoctrine} factors through the forgetful functor $\DL\to\SL$, where $\DL$ is the category of bounded distributive lattices and bounded lattice homomorphisms. Hence, we get a functor
\begin{equation*}
\Sub\colon \C^{\op}\to \DL .
\end{equation*}
We mention in passing that this functor is at the base of the theory of (coherent) hyperdoctrines, a fundamental tool in the categorical semantics of predicate logic~\cite{Lawvere2006}.

%%%%%%%%%%%%%%%%%%%%%%%%%%%%%%%%%%%%
\section{The topological representation}\label{s:top-representation}

Let $\C$ be a category admitting a terminal object $\1$. Throughout, for any object $X$ of $\C$, we call a morphism \[\1\to X\] in the category $\C$ a \emph{point} of $X$. (Thus, for example, a point $x$ of a topological space $X$ is identified with the continuous map $\{*\}\to X$ from the one-point space which selects~$x$.) Points are usually called (\emph{global}) \emph{elements} in category theory.
Each point $\1\to X$ is a section of the unique morphism $!\colon X\to \1$, hence a monomorphism. It follows that every point of $X$ belongs to the poset of subobjects $\Sub{X}$.
We can define a functor
\begin{align}\label{eq:functor-pt}
\pt=\Hom_{\C}(\1,-)\colon \C\to \Set
\end{align}
taking $X$ to the set $\pt{X}$ of its points. 

The aim of this section is to provide sufficient conditions on $\C$, so that the functor $\pt\colon \C\to\Set$ lifts to a faithful functor into the category $\Top$ of topological spaces and continuous maps, yielding a topological representation of $\C$ (cf.\ Theorem~\ref{t:Spec-topological-lifting} below). To achieve this aim we prepare several lemmas.
In seeking a representation of the objects of $\C$ by means of their points, it is useful to assume that the functor $\pt\colon \C\to\Set$ is faithful:
\begin{definition}
A category $\C$ admitting a terminal object $\1$ is \emph{well-pointed} if the functor $\pt$ in~\eqref{eq:functor-pt} is faithful. That is, for any two distinct morphisms $f,g\colon X\rightrightarrows Y$ in $\C$, there is a point $p\colon \1\to X$ such that $f\circ p\neq g\circ p$.
\end{definition}

A coherent category may have non-trivial \emph{subterminal objects}, i.e.\ objects $X\not\cong \0,\1$ such that the unique morphism $X\to \1$ is monic. Such objects have no points, for otherwise the unique morphism $X\to \1$ would be both a monomorphism and a retraction, whence an isomorphism. 
\begin{lemma}\label{l:subterminals}
In a coherent category the following conditions are equivalent:
\begin{enumerate}
\item there are no non-trivial subterminal objects;
\item for every $X\not\cong\0$, the unique morphism $!\colon X\to\1$ is a regular epimorphism.
\end{enumerate}
\end{lemma}
\begin{proof}
Suppose there are no non-trivial subterminal objects and consider the (regular epi, mono) factorisation $X\epi S\mono \1$ of the unique morphism $!\colon X\to\1$. Then, either $S\cong \0$ or $S\cong \1$. Since the initial object is strict by Lemma~\ref{l:cohe-strict-initial}, if $X\not\cong\0$ then it must be $S\cong \1$. Therefore, the unique morphism $X\to \1$ is a regular epi. 
Conversely, assume $X\to\1$ is a regular epimorphism whenever $X\not\cong \0$. If $X$ is not initial and the unique morphism $X\to \1$ is monic, then $X\cong \1$ (just recall that a mono which is also a regular epi must be an isomorphism). Hence, there are no non-trivial subterminal objects.
\end{proof}
Since we aim to capture a classical notion of point, we should ensure that the category at hand has no non-trivial subterminal objects. In fact, we will impose a stronger condition: namely, that any non-initial object admits a point. The points of an object $X$ being exactly the sections of the unique morphism $X\to\1$, this amounts to saying that $!\colon X\to\1$ is a retraction (hence, a regular epimorphism) whenever $X\not\cong\0$. 
\begin{example}
Let $\C=\BA^\op$ be the opposite of the category of Boolean algebras and their homomorphisms. The statement that for every non-initial object $X$ in $\C$ the unique morphism $X\to\1$ is a retraction is equivalent to saying that every non-trivial Boolean algebra admits a maximal ideal. This is known as the \emph{Maximal Ideal Theorem}.
\end{example}
Lemma~\ref{l:cohe-strict-initial} implies that in a coherent category we have $\0\cong \1$ if, and only if, any two objects are isomorphic, i.e.\ the category is equivalent to the terminal category with only one object and one morphism. Thus, if $\0\cong\1$, we say that the category is \emph{trivial}. 
For the remainder of the section, we work with a fixed category $\X$ satisfying the following properties.
\begin{assumption}\label{assumpt:s3}
The category $\X$ is a non-trivial well-pointed coherent category in which the unique morphism $X\to\1$ is a retraction for every $X\not\cong\0$.
\end{assumption}
We note in passing that, if $X$ is an object of $\X$ such that the copower $\sum_{\pt{X}}{\1}$ exists in $\X$, then the canonical morphism
\[
\sum_{\pt{X}}{\1}\to X
\]
is an epimorphism by well-pointedness of $\X$. That is, $X$ is an epimorphic image of the coproduct of its points.
In view of the discussion above, the next lemma is immediate.
\begin{lemma}\label{l:pt-well-defined-and-faithful}
The functor $\pt\colon \X\to\Set$ from~\eqref{eq:functor-pt} is faithful.
\end{lemma}
Recall that an object $X$ in a coherent category has always two (possibly non-distinct) trivial subobjects, namely the unique morphism $\0\to X$ and the identity $X\to X$. If these are the only subobjects of $X$, we say that $X$ \emph{has no non-trivial subobjects}. If $X$ is a subobject of $Y$, then $X$ is an \emph{atom} of $\Sub{Y}$ if, and only if, $X\not\cong\0$ and it has no non-trivial subobjects.
\begin{lemma}\label{l:points-ex-mono}
The following statements hold:
\begin{enumerate}
\item for every $X$ in $\X$, the atoms of the lattice $\Sub{X}$ are precisely the points of $X$; 
\item the functor $\pt\colon \X\to \Set$ preserves regular epis, i.e., $\1$ is regular projective.
\end{enumerate}
\end{lemma}
\begin{proof}
For item $1$, we must prove that the terminal object $\1$ is the unique non-initial object of $\X$ which has no non-trivial subobjects. By Lemma~\ref{l:subterminals}, $\1$ has no non-trivial subobjects. Now, suppose $X$ is an object of $\X$ which admits no non-trivial subobjects. If $X$ is not initial, it has a point $\1\to X$. The latter is a section, whence a monomorphism. Since $\0\ncong\1$, and $X$ has no non-trivial subobjects, we conclude that $X\cong \1$.

For item $2$, let $f\colon X\epi Y$ be a regular epimorphism in $\X$. If $Y\cong\0$, then $f$ is an isomorphism by Lemma~\ref{l:cohe-strict-initial}. Thus $\pt{f}$ is an isomorphism, whence a regular epi. If $Y\not\cong\0$, let $p\colon\1\to Y$ be an arbitrary point of $Y$. We must exhibit $q\in \pt{X}$ such that $\pt{f}(q)=p$. Consider the following pullback square.
\[\begin{tikzcd}
Z \arrow{r}{!} \arrow{d}[swap]{g}  \arrow[dr, phantom, "\lrcorner", very near start] & \1 \arrow{d}{p} \\
X \arrow[twoheadrightarrow]{r}{f} & Y
\end{tikzcd}\]
Since regular epis in $\X$ are pullback stable, $!\colon Z\to \1$ is a regular epi. Note that $Z\ncong\0$, because the unique 
morphism $\0\to\1$ is mono and $\0\ncong \1$. Therefore, $Z$ admits a point $q'\colon \1\to Z$. Defining $q\in\pt{X}$ as 
the composition $g\circ q'\colon\1\to X$ yields 
\[
\pt{f}(q)=f\circ q=f\circ g\circ q'= p\circ {!} \circ q' =p,
\] 
as was to be shown.
\end{proof}
Given an object $X$ of $\X$ and a subobject $S\in\Sub{X}$, define the set
\begin{equation*}
\Vs(S)=\left\{p\colon \1\to X\mid p \ \text{factors through the subobject} \ S\mono X\right\}
\end{equation*}
of all points of $X$ which ``belong to $S$''. Clearly, $\Vs(S)\cong \pt{S}$. Conversely, we would like to be able to define a subobject of $X$ induced by the choice of a subset of points of $X$. Note that the operator $\Vs\colon \Sub{X} \to \P(\pt{X})$ preserves all infima existing in $\Sub{X}$. If the poset of subobjects $\Sub{X}$ is complete then $\Vs$ has a lower adjoint $\Cg\colon \P(\pt{X})\to \Sub{X}$. This is defined by setting, for any subset $T\subseteq \pt{X}$, 
\begin{equation*}
\Cg(T)=\bigwedge{\left\{S\in\Sub{X}\mid \text{each} \ p\in T \ \text{factors through} \ S\right\}}.
\end{equation*}
That is, $\Cg(T)$ is the smallest  subobject of $X$ which ``contains (all the points in) $T$''. The adjunction
\begin{equation*}
\begin{tikzcd}
\P(\pt{X}) \arrow[yshift=-5pt]{rr}[swap]{\Cg}  & {\scriptscriptstyle \top} & \arrow[yshift=5pt]{ll}[swap]{\Vs} \Sub{X}
\end{tikzcd}
\end{equation*}
induces a closure operator $\cl_X=\Vs\circ\Cg$ on $\P(\pt{X})$. To improve readability, we omit reference to the object $X$ and write $\cl$ instead of $\cl_X$. For the next lemma, recall that a \emph{mono-complete} category is one in which every poset of subobjects is complete.
\begin{lemma}\label{l:pseudo-cc}
Assume $\X$ is mono-complete.
For every morphism $f\colon X \to Y$ in $\X$ and every $T\in\P(\pt{X})$, \[\pt{f}(\cl{T})= \cl{(\pt{f}(T))}.\] 
\end{lemma}
\begin{proof}
Fix an arbitrary $T\in\P(\pt{X})$. We first prove that $\pt{f}(\cl{T})\subseteq \cl{(\pt{f}(T))}$. Let $q\in\pt{f}(\cl{T})$, i.e.\ $q=\pt{f}(p)$ for some $p\in\pt{X}$ which belongs to all the subobjects of $X$ which contain all the points in $T$. We must prove that $q$ belongs to every subobject of $Y$ containing all the points of the form $f\circ p'$, with $p'\in T$. Let $S$ be a subobject of $Y$ satisfying the latter property and consider the following pullback square in $\X$.
\[\begin{tikzcd}
f^*(S) \arrow{r} \arrow[hookrightarrow]{d} \arrow[dr, phantom, "\lrcorner", very near start] & S \arrow[hookrightarrow]{d} \\
X \arrow{r}{f} & Y 
\end{tikzcd}\]
By the universal property of the pullback, $f^*(S)$ contains all the points in $T$. Hence $p\in f^*(S)$. It follows that $q=f\circ p\in S$, as was to be proved.

To show that $\pt{f}(\cl{T})\supseteq \cl{(\pt{f}(T))}$, suppose $q\in \cl(\pt{f}(T))$. That is, $q$ is a point of $Y$ which belongs to all the subobjects of $Y$ which contain all the points of the form $\pt{f}(p)$, with $p\in T$. We must prove that $q\in \pt{f}(\cl{T})$. 
Recall that
\begin{align}\label{eq:image-as-meet}
\exists_f(\Cg(T))=\bigwedge{\{S\in\Sub{Y}\mid \Cg(T)\leq f^*(S)\}}.
\end{align}
Now, if $S$ is an arbitrary subobject of $Y$ satisfying $\Cg(T)\leq f^*(S)$, every point of $T$ must belong to $f^*(S)$. Thus, $S$ contains every point of the form $\pt{f}(p)$ for $p\in T$, and so $q\leq S$. By equation~\eqref{eq:image-as-meet}, we have $q\leq  \exists_f(\Cg(T))$. To conclude, it is enough to show that
\begin{align*}
\Vs(\exists_f(\Cg(T)))=\pt{f}(\cl{T}).
\end{align*}
Let $e\colon \Cg(T)\epi \exists_f(\Cg(T))$ be the canonical regular epi. Then, item $2$ in Lemma~\ref{l:points-ex-mono} applies to show that $\pt{e}$ is surjective. Therefore, $\Vs(\exists_f(\Cg(T)))=\pt{f}(\cl{T})$.
\end{proof}
It turns out that the closure operators $\cl$ induce topologies on the sets of points of the objects of $\X$:
\begin{proposition}\label{p:topological-operator}
Let $X$ in $\X$ be such that the lattice $\Sub{X}$ is complete. The closure operator $\cl$ on $\P(\pt{X})$ is topological, i.e.\ it preserves finite unions.
\end{proposition}
\begin{proof}
Recall that $\cl=\Vs\circ\Cg$. 
The operator $\Cg$ preserves arbitrary joins because it is lower adjoint. Hence, it is enough to show that $\Vs$ preserves finite joins. Since $\X$ is non-trivial, $\Vs(\0)=\emptyset$. Now, let $S_1,S_2\in \Sub{X}$ and pick a point $p\in\pt{X}$. The latter is an atom of $\Sub{X}$ by Lemma~\ref{l:points-ex-mono}. Since $\Sub{X}$ is a distributive lattice by Lemma~\ref{l:posets-coherent-distributive-lattice}, and atoms in a distributive lattice are always join-prime, we conclude that $p\leq S_1\vee S_2$ if, and only if, either $p\leq S_1$ or $p\leq S_2$. That is, $\Vs(S_1\vee S_2)=\Vs(S_1)\cup \Vs(S_2)$.
\end{proof}
If $\X$ is mono-complete, the previous proposition entails that, for every object $X$ of $\X$, the set $\pt{X}$ admits a topology whose closed sets are the fixed points for the operator~$\cl$. Write $\Spec{X}$ for the ensuing topological space and notice that this is a $\mathrm{T}_1$-space. For every morphism $f\colon X\to Y$ in $\X$, by the well-known characterisation of continuity in terms of closure operators, the inclusion $\pt{f}(\cl{T})\subseteq\cl{(\pt{f}(T))}$ in Lemma~\ref{l:pseudo-cc} implies that $\pt{f}\colon \Spec{X}\to\Spec{Y}$ is a continuous function. Hence, setting $\Spec{f}=\pt{f}$, the functor $\pt\colon \X\to \Set$ lifts to a functor
\begin{align}\label{eq:functor-Spec}
\Spec\colon \X\to \Top
\end{align}
into the category of topological spaces and continuous maps. Write $|-|\colon \Top \to \Set$ for the underlying-set functor. Because the functor $\pt$ is faithful (Lemma~\ref{l:pt-well-defined-and-faithful}) and the next diagram commutes, we conclude that $\Spec\colon \X\to \Top$ is a faithful functor.
\[\begin{tikzcd}[column sep=12pt]
{} & \X \arrow{dr}{\Spec} \arrow{dl}[swap]{\pt} & \\
\Set & & \Top \arrow{ll}[swap]{|-|}
\end{tikzcd}\]
Therefore, we have proved the following theorem.
\begin{theorem}\label{t:Spec-topological-lifting}
If $\X$ is mono-complete, the functor of points $\pt\colon \X\to \Set$ lifts to a faithful functor $\Spec\colon \X\to \Top$.
\end{theorem}
\begin{remark}\label{rm:Spec-closed-maps}
Suppose $\X$ is mono-complete. For any morphism $f\colon X\to Y$ in $\X$, in view of the inclusion $\pt{f}(\cl{T})\supseteq\cl{(\pt{f}(T))}$ in Lemma~\ref{l:pseudo-cc}, the continuous function $\Spec{f}\colon \Spec{X}\to\Spec{Y}$ is closed. That is, it sends closed sets to closed sets. 
\end{remark}
\begin{remark}
We briefly comment on the topological representation of the category $\X$ obtained in Proposition~\ref{p:topological-operator} and Theorem~\ref{t:Spec-topological-lifting}. If $\X=\KH$ and $X$ is a compact Hausdorff space, then the operator $\cl$ on $\P(|X|)$ sends a subset $T\subseteq X$ to the intersection of all closed subspaces of $X$ containing $T$. Just recall that subobjects in $\KH$ can be identified with closed subspaces. Hence, $\cl\colon \P(|X|)\to \P(|X|)$ is the usual topological closure operator associated with $X$, and the space $\Spec{X}$ is homeomorphic to $X$. A similar reasoning applies to the case where $\X=\BStone$.  

The adjunction $\Cg\dashv \Vs\colon \Sub{X} \to \P(\pt{X})$ could be set up for an arbitrary topological space $X$, although the category $\Top$ does not satisfy all the conditions in Assumption~\ref{assumpt:s3}. However, the ensuing operator $\cl$ would not coincide with the usual topological closure. A first hurdle is that subobjects in $\Top$ correspond to (equivalence classes of) continuous injections, which cannot always be identified with subspaces. Even if we restricted ourselves to \emph{regular} subobjects---which correspond to subspaces---the operator $\cl$ would simply be the identity map $\P(|X|)\to\P(|X|)$.
\end{remark}
For every object $X$ of $\X$, the co-restriction of the map $\Vs\colon \Sub{X}\to\P(\pt{X})$ to the set of fixed points of $\cl$ yields a surjective lattice homomorphism 
\begin{align}\label{eq:sub-KSpec}
\Vs\colon \Sub{X}\epi\K(\Spec{X}),
\end{align} 
where $\K(\Spec{X})$ denotes the lattice of closed subsets of $\Spec{X}$.
We conclude this section by showing that this map is a lattice isomorphism if, and only if, $\Sub{X}$ is an \emph{atomic} lattice, i.e.\ every element of $\Sub{X}$ is the supremum of the atoms below it.
\begin{lemma}\label{l:atomic-fixed-point}
Let $X$ in $\X$ be such that the lattice $\Sub{X}$ is complete. The map in~\eqref{eq:sub-KSpec} is a lattice isomorphism between $\Sub{X}$ and $\K(\Spec{X})$ if, and only if, $\Sub{X}$ is atomic. 
\end{lemma}
\begin{proof}
Since the space $\Spec{X}$ is $\mathrm{T}_1$, the lattice $\K(\Spec{X})$ is atomic. Hence, if there exists a lattice isomorphism between $\K(\Spec{X})$ and $\Sub{X}$, the latter must be atomic.

Conversely, suppose $\Sub{X}$ is atomic. We will prove that, for each $S\in\Sub{X}$, $\Cg\circ\Vs(S)=S$. It will follow that $\Vs\colon\Sub{X}\epi\K(\Spec{X})$ is injective, whence a lattice isomorphism. Let $S\in \Sub{X}$ be an arbitrary subobject. Clearly, we have $\Cg\circ\Vs(S)\leq S$. For the converse inequality, we must prove that $S\leq S'$ whenever $S'\in\Sub{X}$ is such that every point of $X$ which factors through $S$ factors also through $S'$. In view of item $1$ in Lemma~\ref{l:points-ex-mono}, this holds provided $\Sub{X}$ is atomic.
\end{proof}
%

%%%%%%%%%%%%%%%%%%%%%%%%%%%%%%%%%%%%
\section{Filtrality}\label{s:filtrality}
For any bounded distributive lattice $L$, let $\Ct(L)$ denote the subset of $L$ consisting of the complemented elements of $L$. Thus, $x\in\Ct(L)$ if, and only if, there exists an element $\neg x \in L$ such that $x\wedge \neg x=0$ and $x\vee\neg x=1$, where $0$ and $1$ are the bottom and top elements of $L$, respectively. Such an element $\neg x$ is unique  and lies in $\Ct(L)$, whenever it exists. Then $\Ct(L)$, the \emph{Boolean center} of $L$, is the inclusion-largest Boolean algebra embedded as a bounded sublattice in $L$.

Let $\F(L)$ be the set of all non-empty filters of $L$ partially ordered by reverse inclusion. Then, $\F(L)$ is again a bounded distributive lattice. For $x\in L$, define the principal filter $\mathord{\uparrow}{x}=\left\{y\in L\mid x\leq y\right\}$. The  function 
\[
\eta_L\colon L\to \F(L), \ \ x\mapsto \mathord{\uparrow}{x}
\] is a  monomorphism of bounded lattices known as the filter completion of $L$.
Further, consider the contraction map $\chi_L\colon \F(L)\to \F(\Ct(L))$ sending $F\in\F(L)$ to $F\cap \Ct(L)\in \F(\Ct(L))$. It is clear that $\chi_L$ is well-defined, and monotone. 
We are  interested in  the composition
\[\begin{tikzcd}
\phi\colon L \arrow{r}{\eta_L} & \F(L) \arrow{r}{\chi_L} & \F(\Ct(L)), \ \ x \mapsto \mathord{\uparrow}{x} \cap \Ct(L).
\end{tikzcd}\]
We call the bounded distributive lattice $L$  \emph{filtral} precisely when this map $\phi$ is an isomorphism of posets (and hence of bounded lattices).

If $X$ is an object of a coherent category, let us write $\B(X)$ as a short-hand for the unwieldy $\Ct(\Sub{X})$.
\begin{definition}\label{def:filtrality}
An object $X$ of a coherent category $\C$ is \emph{filtral} if $\Sub{X}$ is a filtral lattice. The category $\C$ is said to be \emph{filtral} if each of its objects is covered by a filtral one, i.e.\ for every $Y$ in $\C$ there is a regular epimorphism $X\epi Y$ with $X$ filtral.
\end{definition}
Recall that a \emph{co-frame} is a complete lattice $L$ satisfying, for every set $I$, the infinite distributive law 
\begin{equation}\label{eq:co-frame}
\forall \{a\}\cup\{b_i\mid i\in I\}\subseteq L, \ \ a\vee \big(\bigwedge_{i\in I}{b_i}\big)=\bigwedge_{i\in I}{(a\vee b_i)}.
\end{equation}
More frequent is the notion of frame, the order-dual of a co-frame. It is well known that $I(L)$, the collection of ideals of a bounded distributive lattice $L$, is a frame when ordered by inclusion. Further, the assignment $L\mapsto I(L)$ extends to a functor from the category of bounded distributive lattices to the category of frames (with morphisms the functions preserving finite infima and arbitrary suprema) which is left adjoint to the obvious forgetful functor. See e.g.~\cite[Corollary II.2.11]{Johnstone1986}. It is not difficult to see that $\F(L)\cong I(L^\op)^\op$, whence the filter completion of $L$ is a co-frame. The assignment $L\mapsto \F(L)$ extends to a functor from the category of bounded distributive lattices to the category of co-frames (with morphisms the functions preserving finite suprema and arbitrary infima) which is left adjoint to the obvious forgetful functor. The components of the  unit of this adjunction are the filter completions $\eta_L\colon L\to \F(L)$. In more detail, whenever $L,M$ are bounded distributive lattices and $h\colon L\to M$ is a homomorphism of bounded lattices, we have a co-frame homomorphism 
\[
\F(h)\colon \F(L)\to \F(M), \ \ F\mapsto \mathord{\uparrow}{h(F)}=\{m\in M\mid \exists l\in F \ \text{s.t.} \ h(l)\leq m\}.
\]
Note that $\F(h)$ admits a lower adjoint, namely $h^{-1}\colon \F(M)\to \F(L)$. In these terms, the contraction map $\chi_L\colon \F(L)\to \F(\Ct(L))$ is the lower adjoint to the co-frame homomorphism $\F(i)\colon \F(\Ct(L))\to \F(L)$, where $i\colon \Ct(L)\mono L$ is the inclusion of the Boolean center of $L$. It is not difficult to see that the map $\phi\colon \Sub{X}\to \F(\B(X))$ preserves all suprema which exist in $\Sub{X}$. Therefore, whenever $\Sub{X}$ is complete, $X$ is filtral if, and only if, $\phi$ is a bijection.

As we shall see in the next lemma, in those coherent categories which admit a topological representation as in the previous section, the notion of filtral object can be rephrased without explicit reference to the isomorphism $\phi\colon \Sub{X}\to \F(\B(X))$.
\begin{lemma}\label{l:charact-filtral-object}
Let $\X$ be a non-trivial well-pointed coherent category in which the unique morphism from any non-initial object to the terminal object is a retraction. For any $X$ in $\X$, the following statements are equivalent:
\begin{enumerate}
\item $X$ is filtral;
\item there is an isomorphism of bounded lattices between $\Sub{X}$ and $\F(\B(X))$;
\item $\Sub{X}$ is a complete and atomic lattice, and $\Spec{X}$ is a Stone space.
\end{enumerate} 
\end{lemma}
\begin{proof}
$1\Rightarrow 2$. Clear. 

$2\Rightarrow 3$. Suppose there is an isomorphisms of bounded lattices $\Sub{X}\cong \F(\B(X))$. The lattice $\F(\B(X))$ is a co-frame, hence complete. Further it is atomic, because every filter of a Boolean algebra is the intersection of all the ultrafilters extending it. Thus, $\Sub{X}$ is also a complete and atomic lattice. It remains to prove that $\Spec{X}$ is a Stone space. By Stone duality for Boolean algebras~\cite{Stone1936}, the lattice $\F(\B(X))$ is isomorphic to the lattice of closed sets of the Stone space $\mathcal{S}$ dual to the Boolean algebra $\B(X)$.
On the other hand, Lemma~\ref{l:atomic-fixed-point} entails that $\Sub{X}$ is isomorphic to $\K(\Spec{X})$, the lattice of closed sets of $\Spec{X}$. Therefore, $\K(\Spec{X})$ is isomorphic to $\K(\mathcal{S})$. It is well-known that two $\mathrm{T}_1$-spaces with isomorphic lattices of closed sets are homeomorphic~\cite{Thron1962}. Thus, $\Spec{X}$ is homeomorphic to the Stone space $\mathcal{S}$. 

$3\Rightarrow 1$. Assume $\Sub{X}$ is a complete lattice. Then, the space $\Spec{X}$ is well-defined by Proposition~\ref{p:topological-operator}. If the latter is a Stone space, the Boolean center of $\K(\Spec{X})$ coincides with the Boolean algebra of clopen subsets of $\Spec{X}$, denoted by $\textnormal{Cl}(\Spec{X})$. We claim that the following diagram commutes, where $\Vs$ is the surjective homomorphism of bounded lattices from equation~\eqref{eq:sub-KSpec}.
\[\begin{tikzcd}[column sep=45pt, row sep=25pt]
\Sub{X} \arrow{r}{\eta_{\Sub{X}}} \arrow{d}[swap]{\Vs} & \F(\Sub{X}) \arrow{r}{\chi_{\Sub{X}}} \arrow{d}[swap]{\F(\Vs)} & \F(\B(X)) \arrow{d}{\F(\Vs\restriction\B(X))} \\
\K(\Spec{X}) \arrow{r}{\eta_{\K(\Spec{X})}} & \F(\K(\Spec{X})) \arrow{r}{\chi_{\K(\Spec{X})}} & \F(\textnormal{Cl}(\Spec{X}))
\end{tikzcd}\]
The leftmost square commutes by naturality of the transformation $\eta$. The commutativity of the rightmost square is an instance of a Beck-Chevalley property. We give a direct proof. For every $F\in \F(\Sub{X})$,
\begin{align*}
\left(\F(\Vs\restriction\B(X))\circ \chi_{\Sub{X}}\right)(F)&=\F(\Vs\restriction\B(X))(F\cap \B(X)) \\
&=\{U\in \textnormal{Cl}(\Spec{X})\mid \exists S\in F\cap \B(X) \ \text{s.t.} \ \Vs(S)\subseteq U\} \\
&\subseteq \{U\in \textnormal{Cl}(\Spec{X})\mid \exists S\in F \ \text{s.t.} \ \Vs(S)\subseteq U\} \\
&=\F(\Vs)(F)\cap \textnormal{Cl}(\Spec{X}) \\
&=\left(\chi_{\K(\Spec{X})}\circ\F(\Vs)\right)(F).
\end{align*}
To settle the converse inclusion, suppose that $U\in \textnormal{Cl}(\Spec{X})$ satisfies $\Vs(S)\subseteq U$ for some $S\in F$. If $\Sub{X}$ is atomic, then $\Vs\colon \Sub{X}\to \K(\Spec{X})$ is an isomorphism by Lemma~\ref{l:atomic-fixed-point}. Whence, there is $T\in \B(X)$ such that $\Vs(T)=U$. Then, $\Vs(S)\subseteq \Vs(T)$ implies $S\leq T$. Therefore $T\in F\cap \B(X)$ and $\Vs(T)=U$. 

This shows that the outer rectangle in the diagram above commutes. To conclude the proof, note that the vertical arrows are bounded lattice isomorphisms. Hence, the bottom row is a bounded lattice isomorphism if, and only if, so is the top row. That is, if and only if $X$ is a filtral object. In turn, the bottom row sends a closed subset $K\subseteq\Spec{X}$ to the filter consisting of all the clopen subsets of $\Spec{X}$ containing $K$. It is well-known that this map is a bounded lattice isomorphism whenever $\Spec{X}$ is a Stone space.
\end{proof}
In view of the previous lemma, we can regard filtral objects in a coherent category as an abstraction of zero-dimensional compact Hausdorff spaces. The next example shows that the filtral objects in $\KH$ are precisely the Stone space.
\begin{example}\label{ex:KH-filtral}
Note that the category $\KH$ satisfies the assumptions of Lemma~\ref{l:charact-filtral-object}. Moreover, for any compact Hausdorff space $X$, its lattice of closed subsets $\K(X)\cong \Sub{X}$ is complete and atomic. Hence, Lemma~\ref{l:charact-filtral-object} entails that $X$ is filtral in $\KH$ if, and only if, $\Spec{X}$ is a Stone space. Since the spaces $X$ and $\Spec{X}$ are homeomorphic, we  conclude that the filtral objects in the category $\KH$ are precisely the Stone spaces.
Also, because every compact Hausdorff space is the continuous image of a Stone space (e.g., of the Stone-{\v C}ech compactification of its underlying set equipped with the discrete topology), the category $\KH$ is filtral.
It follows that its full subcategory $\BStone$ is also filtral, and each of its objects is filtral.
\end{example}
An easy, yet useful observation is that every filtral category is mono-complete:
\begin{lemma}\label{l:filtral-co-frame}
In a filtral category every poset of subobjects is a co-frame, and the pullback functors are co-frame homomorphisms. In particular, a filtral category is mono-complete.
\end{lemma}
\begin{proof}
Note that pullback functors preserve---in addition to finite suprema---also all existing infima, because they are upper adjoint. Thus, it suffices to show that each poset of subobjects is a co-frame.

First, observe that in any coherent category the pullback functor associated with a regular epimorphism is an injective lattice homomorphism. 
To see this, consider a regular epimorphism $\epsilon\colon W\epi Z$ and subobjects $S_1,S_2\in \Sub{Z}$. Then, for $i\in\{1,2\}$, $S_i$ coincides with the image of the morphism $\epsilon \circ m_i\colon \epsilon^*(S_i)\to Z$, where $m_i\colon \epsilon^*(S_i)\mono W$ is the canonical monomorphism. Therefore, $\epsilon^*(S_1)\cong\epsilon^*(S_2)$ implies $S_1\cong S_2$. 

Now, let $Y$ be any object and $\epsilon\colon X\epi Y$ a regular epimorphism with $X$ filtral.
By the previous observation, the lattice homomorphism $\epsilon^*\colon \Sub{Y}\mono\Sub{X}$ is injective. The lattice $\Sub{X}\cong \F(\B(X))$ is a co-frame, whence complete. Given $\{S_i\mid i\in I\}\subseteq\Sub{Y}$, 
\[
\exists_{\epsilon}\big(\bigvee\{\epsilon^*(S_i)\mid i\in I\}\big)
\]
is readily seen to be the supremum of the $S_i$'s in $\Sub{Y}$, where $\exists_{\epsilon}$ is the lower adjoint of~$\epsilon^*$. Thus, $\Sub{Y}$ is also a complete lattice.
Further, since $\Sub{X}$ satisfies the co-frame law~\eqref{eq:co-frame}, so does $\Sub{Y}$ (just recall that $\epsilon^*\colon \Sub{Y}\mono\Sub{X}$ preserves all infima).
\end{proof}

For the rest of this section, we fix a category $\X$ satisfying the following properties. 
\begin{assumption}\label{assump:s4}
The category $\X$ is a non-trivial, well-pointed coherent category. Further, the unique morphism $X\to\1$ is a retraction for every $X\not\cong\0$.
\end{assumption}
In view of Lemma~\ref{l:filtral-co-frame}, the functor $\Spec\colon \X\to \Top$ from~\eqref{eq:functor-Spec} is well-defined whenever $\X$ is filtral. The following proposition shows that, in this situation, the functor $\Spec$ takes values in the category of compact Hausdorff spaces. 
\begin{proposition}\label{p:filtrality}
If the category $\X$ is filtral then $\Spec{X}$ is a compact Hausdorff space for every $X$ in $\X$.
\end{proposition}
\begin{proof}
Let $X$ be an arbitrary object of $\X$. If the category $\X$ is filtral, there is a filtral object $Y$ in $\X$ and a regular epimorphism $Y\epi X$.
In view of Lemma~\ref{l:charact-filtral-object}, $\Spec{Y}$ is a Stone space.
By item $2$ in Lemma~\ref{l:points-ex-mono} and Remark~\ref{rm:Spec-closed-maps}, the image under the functor $\Spec$ of the regular epimorphism $Y\epi X$ is a continuous and closed surjection. Since $\Spec{Y}$ is compact and Hausdorff, so is $\Spec{X}$.
\end{proof}
Proposition~\ref{p:filtrality} entails at once that, whenever $\X$ is filtral, the functor $\Spec\colon \X\to\Top$ co-restricts to a functor
\begin{align}\label{eq:functor-Spec-to-KH}
\Spec\colon \X\to \KH.
\end{align}
We conclude this section by recording another consequence of Proposition~\ref{p:filtrality}.
\begin{lemma}\label{l:spec-pres-limits}
If $\X$ is filtral, then the functor $\Spec\colon \X\to\KH$ preserves all limits that exist in $\X$.
\end{lemma}
\begin{proof}
Consider the following commutative diagram of functors.
\[\begin{tikzcd}[column sep=12pt]
{} & \X \arrow{dl}[swap]{\Spec} \arrow{dr}{\pt} & \\
\KH \arrow{rr}{|-|} & & \Set
\end{tikzcd}\]
The underlying-set functor $|-|\colon \KH\to\Set$ preserves all limits because it is represented by the one-point space. Further, it is \emph{conservative} (i.e., it reflects isomorphisms), since any continuous bijection between compact Hausdorff spaces is a homeomorphism. A conservative functor reflects all the limits it preserves, whence $|-|$ reflects all limits. Since a limit in $\X$ is preserved by $\pt$, it must also be preserved by $\Spec$.
\end{proof}
%

%%%%%%%%%%%%%%%%%%%%%%%%%%%%%%%%%%%%
\section{The main result}\label{s:main}
The aim of this section is to prove our main result, that is the following characterisation of the category $\KH$ of compact Hausdorff spaces and continuous maps. (Recall that all categories under consideration are assumed to be well-powered.)
\begin{theorem}\label{t:main-characterisation-bis}
Up to equivalence, $\KH$ is the unique non-trivial pretopos which is well-pointed, filtral and admits all set-indexed copowers of its terminal object. 
\end{theorem}
 We recall the definitions and facts needed to prove the previous theorem.
 A category $\C$ is (\emph{Barr}) \emph{exact} provided it is regular and every internal equivalence relation in $\C$ is \emph{effective}, i.e.\ it coincides with the kernel pair of its coequaliser (see, e.g.,~\cite{BGvO71} or~\cite[Sections~2.5--2.6]{Borceux2}). Exact categories are those in which there is a good correspondence between congruences (i.e., internal equivalence relations) and quotients (i.e., coequalisers). All varieties of Birkhoff algebras and, more generally, all categories which monadic over $\Set$, are exact~\cite[Theorems 3.5.4, 4.3.5]{Borceux2}. Roughly speaking, a pretopos is an exact category in which finite coproducts exist and are ``well-behaved''. The latter property is formalised by the notion of extensivity.
 
\begin{definition}\label{def:extensivity}
A category $\C$ is \emph{extensive} provided it has finite coproducts and the canonical functor \[(\C/X_1) \times (\C/X_2)\to \C/(X_1+X_2)\] is an equivalence for every $X_1,X_2$ in $\C$. 
\end{definition}
In the presence of enough limits, a more intuitive reformulation of this notion is available.
Given two objects $X_1,X_2$ in $\C$, the coproduct $X_1+X_2$ is \emph{universal} if the pullback of the coproduct diagram $X_1\to X_1+X_2 \leftarrow X_2$ along any morphism yields a coproduct diagram. Moreover, recall that the coproduct $X_1+X_2$ is \emph{disjoint} if pulling back a coproduct injection along the other yields the initial object of $\C$.
 \begin{lemma}[{\cite[Proposition 2.14]{CLW93}}]
 If $\C$ has finite coproducts and pullbacks along coproduct injections, then it is extensive if and only if finite coproducts in $\C$ are universal and disjoint.
 \end{lemma}
 \begin{definition}
 A \emph{pretopos} is an exact and extensive category.
 \end{definition}
Pretoposes are often defined as positive and effective coherent categories (see, e.g., \cite[A1.4]{Elephant1}). Here, \emph{positive} means that finite coproducts exist and are disjoint, while an \emph{effective} regular category is what has been called an exact category above. It is not difficult to see that the two definitions are equivalent: disjoint coproducts in a coherent category are universal, and an exact extensive category is automatically coherent. We record this fact for future use:
\begin{lemma}\label{l:pretopos-implies-coherent}
A category $\C$ is a pretopos if, and only if, it is a positive and effective coherent category.
\end{lemma}
\begin{example}\label{ex:examples-pretoposes}
We give some examples of categories that are, or are not, pretoposes.
\begin{itemize}
\item $\Set$ is a pretopos. Its full subcategory on the finite sets is also a pretopos. More generally, every elementary topos is a pretopos.
\item $\KH$ is a pretopos. Exactness follows from the fact that $\KH$ is monadic over $\Set$ \cite[Section~5]{Linton66}, but can also be verified directly.
\item $\BStone$ is a positive coherent category but is not exact, hence it is not a pretopos. To see this, let $\beta([0,1])$ denote the Stone-{\v C}ech compactification of the unit interval with the discrete topology. There is an equivalence relation $R\rightrightarrows \beta([0,1])$ in $\BStone$ which identifies two ultrafilters on $[0,1]$ precisely when they have the same limit. The coequaliser of $R$ in $\BStone$ is the unique morphism to the one-point space, whose kernel pair is the improper relation on $\beta([0,1])$. Hence, $\BStone$ is not exact. Its exact completion coincides with $\KH$ (cf.\ Theorem~\ref{t:exact-completion-stone}).
\end{itemize}
\end{example}
The strategy to prove Theorem~\ref{t:main-characterisation-bis} is the following. We first show that, if $\X$ is a pretopos satisfying the properties in the statement of the theorem, then the functor $\Spec\colon \X\to \KH$ is well-defined and preserves a part of the categorical structure of $\X$. Namely, that of coherent category. We then use this information to show that $\Spec$ is an equivalence of categories.
\begin{lemma}\label{l:spec-functor-defined}
If $\X$ is a non-trivial well-pointed pretopos which is filtral, then the functor $\Spec\colon \X\to \KH$ from~\eqref{eq:functor-Spec-to-KH} is well-defined.
\end{lemma}
\begin{proof}
Assume $\X$ is as in the statement. It suffices to verify that the conditions in Assumption~\ref{assump:s4} are satisfied, for then it will follow by Proposition~\ref{p:filtrality} that the functor $\Spec\colon \X\to \KH$ is well-defined. In view of Lemma~\ref{l:pretopos-implies-coherent}, it is enough to show that the unique morphism $X\to\1$ is a retraction whenever $X\not\cong\0$.

Suppose $X$ satisfies $\pt{X}=\emptyset$. By well-pointedness of $\X$ there is an epimorphism $\0\cong\sum_{\emptyset}{\1}\to X$. Because every epimorphism in a pretopos is regular (cf., e.g., \cite[Corollary A.1.4.9]{Elephant1}), the unique morphism $\0\to X$ is both a monomorphism and a regular epimorphism, whence an isomorphism. That is, $X\cong \0$.
\end{proof}
For the next lemma, recall that a \emph{coherent functor} is a functor between coherent categories which preserves finite limits, regular epimorphisms and finite joins of subobjects. 
\begin{lemma}\label{l:spec-functor-coherent}
If $\X$ is a non-trivial well-pointed pretopos which is filtral, then the functor $\Spec\colon \X\to \KH$ is coherent.
\end{lemma}
\begin{proof}
The functor $\Spec\colon\X\to\KH$ preserves finite limits by Lemma~\ref{l:spec-pres-limits}. Regular epis in $\KH$ are simply continuous surjective functions, therefore $\Spec$ preserves regular epis by item $2$ in Lemma~\ref{l:points-ex-mono}. It remains to prove that $\Spec$ preserves finite joins of subobjects. We first prove the following fact:
\begin{claim}\label{claim:spec-fin-coprod}
The functor $\Spec$ preserves finite coproducts.
\end{claim}
\begin{proof}
Since the initial object of $\X$ is strict, we have $\Spec{\0}=\emptyset$. It thus suffices to prove that $\Spec{(X+Y)}\cong \Spec{X}+\Spec{Y}$ whenever $X,Y$ are objects of $\X$. At the level of underlying sets, the obvious function \[\pt{X}+\pt{Y}\to \pt{(X+Y)}\] is injective because finite coproducts in $\X$ are disjoint. On the other hand, surjectivity follows from the universality of coproducts. To prove that this bijection is a homeomorphism, we have to show that every subobject of $X+Y$ splits as the coproduct of a subobject of $X$ and a subobject of $Y$. In turn, this follows again from the universality of finite coproducts in $\X$. Just observe that the pullback of the coproduct $X\to X+Y \leftarrow Y$ along a subobject $S\mono X+Y$ yields a splitting of $S$ of the form $S\cong S_1+S_2$, with $S_1\in\Sub{X}$ and $S_2\in\Sub{Y}$.
\end{proof}
To conclude the proof of the lemma, consider $X$ in $\X$ and two subobjects $S_1\mono X$ and $S_2\mono X$. We want to show that $\Spec{(S_1\vee S_2)}\cong \Spec{S_1}\vee \Spec{S_2}$. Let \[j\colon S_1+ S_2\to X\] be the coproduct of the two subobjects. By the previous claim, the functor $\Spec$ preserves finite coproducts, thus $\Spec{j}\colon \Spec{(S_1+S_2)}\to \Spec{X}$ is the coproduct of the subobjects $\Spec{S_1}\mono \Spec{X}$ and $\Spec{S_2}\mono \Spec{X}$. The subobject $S_1\vee S_2\mono X$ is obtained by taking the image, i.e.\ the (regular epi, mono) factorisation, of $j$. Since $\Spec$ preserves regular epis and monos by the first part of the proof, the image under $\Spec$ of the (regular epi, mono) factorisation of $j$ is the (regular epi, mono) factorisation of $\Spec{j}$. Hence, 
$\Spec{(S_1\vee S_2)}\cong \Spec{S_1}\vee \Spec{S_2}$.
\end{proof}
The last ingredient we need in order to prove Theorem~\ref{t:main-characterisation-bis} is the following proposition, due to Makkai.
Suppose $\C,\D$ are coherent categories. A coherent functor $F\colon \C\to\D$ is \emph{full on subobjects} if, for any $X$ in $\C$, the induced lattice homomorphism $\Sub{X}\to\Sub{FX}$ is surjective (the latter map is well-defined because coherent functors preserve finite limits and, in particular, monomorphisms). The functor $F$ \emph{covers} $\D$ if, for each $Y$ in $\D$, there exist $X$ in $\C$ and a regular epimorphism $FX\epi Y$ in $\D$.
Moreover, $F$ is \emph{conservative} if it reflects isomorphisms.
Finally, a \emph{morphism of pretoposes} is a functor between pretoposes which preserves finite limits, finite coproducts, and coequalisers of internal equivalence relations.
\begin{proposition}[{\cite[Proposition 2.4.4 and Lemma 2.4.6]{Makkai85}}]\label{p:pretopos-morphisms}
The following statements hold:
\begin{enumerate}
\item any coherent functor between pretoposes is a morphism of pretoposes;
\item a morphism of pretoposes is an equivalence if, and only if, it is conservative, full on subobjects, and covers its codomain.
\end{enumerate}
\end{proposition}
We are now ready for the proof of our main result.
\begin{proof}[of Theorem~\ref{t:main-characterisation-bis}]
The category $\KH$ is a non-trivial pretopos which is well-pointed and filtral (cf.\ Examples~\ref{ex:KH-filtral} and~\ref{ex:examples-pretoposes}). Further, for any set $I$, the copower $\sum_{I}{\1}$ in $\KH$ coincides with $\beta(I)$, the Stone-{\v C}ech compactification of the discrete space $I$. Hence, $\KH$ admits all set-indexed copowers of its terminal object. To show that, up to equivalence, it is the unique such category, let $\X$ be a pretopos satisfying these properties.
By Lemma~\ref{l:spec-functor-coherent} and Proposition~\ref{p:pretopos-morphisms}, in order to show that the functor $\Spec\colon \X\to \KH$ is an equivalence, it suffices to prove that it is (i) conservative, (ii) full on subobjects and (iii) it covers $\KH$. 

\begin{enumerate}[wide, labelwidth=!, labelindent=10pt]\renewcommand{\labelenumi}{(\roman{enumi})}
\item The functor $\Spec\colon\X\to\KH$ is faithful because so is $\pt\colon \X\to\Set$ (cf.\ Lemma~\ref{l:pt-well-defined-and-faithful}). Therefore, it reflects both epis and monos. Since a pretopos is balanced (see, e.g.,~\cite[Corollary A.1.4.9]{Elephant1}) $\Spec$ reflects isomorphisms, i.e., it is conservative.
\item Recall that monomorphisms in $\KH$ are inclusions of closed subsets. Whence, $\Spec$ is full on subobjects by equation~\eqref{eq:sub-KSpec}.
\item Consider any compact Hausdorff space $Y$, and write $|Y|$ for its underlying set. By assumption, the $|Y|$-fold copower of $\1$ exists in $\X$. Moreover, by filtrality of $\X$, there is a regular epimorphism $X\epi \sum_{|Y|}{\1}$ in $\X$ with $X$ filtral. We prove the following useful fact:
\end{enumerate}
\begin{claim}
$\P(|Y|)$ is isomorphic to a Boolean subalgebra of $\B(X)$, the Boolean center of $\Sub{X}$.
\end{claim}
\begin{proof}[of Claim]
The pullback functor associated with the regular epi $X\epi \sum_{|Y|}{\1}$ is an injective lattice homomorphism $\Sub\sum_{|Y|}{\1}\mono \Sub{X}$ (cf.\ the proof of Lemma~\ref{l:filtral-co-frame}). Hence, its restriction to the Boolean center of $\Sub{\sum_{|Y|}{\1}}$ yields an injective Boolean algebra homomorphism  $\B(\sum_{|Y|}{\1})\mono \B(X)$. Therefore, it suffices to prove that $\P(|Y|)$ can be identified with a Boolean subalgebra of $\B(\sum_{|Y|}{\1})$.

To this end, consider the map
\[
k\colon \F(\P(|Y|))\to \Sub{\sum_{|Y|}{\1}}, \ \ k(F)=\bigwedge{\{S\in \Sub{\sum_{|Y|}{\1}}\mid \pt{S}\cap |Y|\in F\}}.
\]
We claim that $k$ is a bounded lattice homomorphism. Upon recalling that elements of $\F(\P(|Y|))$ are ordered by reverse inclusion,  for every $F_1,F_2\in \F(\P(|Y|))$ we have 
\begin{align*}
k(F_1\vee F_2)&= \bigwedge{\{S\in \Sub{\sum_{|Y|}{\1}}\mid \pt{S}\cap |Y|\in F_1\cap F_2\}} \\
&= \bigwedge{\{S_1\vee S_2\mid S_1,S_2\in \Sub{\sum_{|Y|}{\1}}, \ \pt{S_1}\cap |Y|\in F_1 \ \text{and} \ \pt{S_2}\cap|Y|\in F_2\}}.
\end{align*}
By Lemma~\ref{l:filtral-co-frame}, the co-frame law~\eqref{eq:co-frame} holds in $\Sub{\sum_{|Y|}{\1}}$. Thus, the latter infimum coincides with
\[
\bigwedge{\{S\in \Sub{\sum_{|Y|}{\1}}\mid \pt{S}\cap |Y|\in F_1\}}\vee \bigwedge{\{S\in \Sub{\sum_{|Y|}{\1}}\mid \pt{S}\cap |Y|\in F_2\}}
\]
which, by definition, is equal to $k(F_1)\vee k(F_2)$.
Further, $k(\P(|Y|))=\0$. This shows that $k$ preserves finite suprema. A straightforward computation shows that $k$ preserves also finite infima, whence it is a bounded lattice homomorphism. The restriction of $k$ to the Boolean center of $\F(\P(|Y|))$, which is isomorphic to $\P(|Y|)$, gives a Boolean algebra homomorphism $\P(|Y|)\to \B\big(\sum_{|Y|}{\1}\big)$ which is injective. Just observe that $T\in \P(|Y|)$ is sent to the complemented subobject 
\[
\bigwedge{\{S\in \Sub{\sum_{|Y|}{\1}}\mid T\subseteq \pt{S}\}}=\Cg(T),
\]
which coincides with the initial object $\0$ if, and only if, $T=\emptyset$.
\end{proof}
Since $X$ is filtral, the space $\Spec{X}$ is homeomorphic to the Stone space dual to $\B(X)$ (cf.\ the proof of Lemma~\ref{l:charact-filtral-object}). Hence, the previous claim entails that the Stone-\v{C}ech compactification $\beta(|Y|)$ of the discrete space $|Y|$ is a continuous image of $\Spec{X}$. In turn, by the universal property of the Stone-{\v C}ech compactification, the identity function $|Y|\to Y$ lifts to a continuous surjection $\beta(|Y|)\epi Y$. Composing the two maps, we obtain a continuous surjection $\Spec{X}\epi Y$. This shows that the functor $\Spec$ covers $\KH$.
\end{proof}
We remark that the hypotheses of Theorem~\ref{t:main-characterisation-bis} are independent from each other. 
Below, we give examples of categories which satisfy all the assumptions of the theorem, except for the one in parentheses:
\begin{itemize}
\item $\{\0\cong\1\}$ (non-trivial);
\item $\{\0\to\1\}$ (extensive);
\item $\BStone$ (exact);
\item $\KH/\{0,1\}$, where $\{0,1\}$ is equipped with the discrete topology (well-pointed);
\item $\Set$ (filtral);
\item $\Setfin$ (admits all set-indexed copowers of its terminal object).
\end{itemize}
The fact that $\KH/\{0,1\}$ satisfies all the hypotheses in Theorem~\ref{t:main-characterisation-bis} except for well-pointedness is a consequence of the following observation.
\begin{lemma}
Let $X$ be any compact Hausdorff space with at least two distinct points. Then the slice category $\KH/X$ is a non-trivial pretopos which is filtral and admits all coproducts, but is not well-pointed. 
\end{lemma}
\begin{proof}
Let $X$ be as in the statement. The slice category $\KH/X$ is easily seen to be extensive (cf.~\cite[Proposition 4.8]{CLW93}). Moreover, any slice category of an exact category is exact. For a proof see, e.g.,~\cite[p.~435]{BB2004}. Thus, $\KH/X$ is a pretopos. The initial and terminal objects in $\KH/X$ are the unique morphism $\emptyset\to X$ and the identity function $X\to X$, respectively. Hence, $\KH/X$ is trivial if, and only if, $X=\emptyset$. This shows that $\KH/X$ is a non-trivial pretopos.

The coproduct in $\KH/X$ of a set of objects $\{Y_i\to X\mid i\in I\}$ is the unique morphism $\sum_{i\in I}{Y_i}\to X$ induced by the universal property of coproducts in $\KH$, whence $\KH/X$ admits all coproducts. To show that $\KH/X$ is filtral, consider an object $f\colon Y\to X$ in $\KH/X$. Since $\KH$ is filtral, there is a filtral compact Hausdorff space $Z$ and a regular epimorphism $\epsilon\colon Z\epi Y$ in $\KH$. The latter is a regular epimorphism in $\KH/X$ from $f\circ \epsilon\colon Z\to X$ to $f\colon Y\to X$. Further, because the lattice of subobjects of $f\circ\epsilon\colon Z\to X$ in $\KH/X$ is isomorphic to the lattice of subobjects of $Z$ in $\KH$, it is not difficult to see that $f\circ \epsilon$ is a filtral object covering $f$. We conclude that $\KH/X$ is a non-trivial pretopos which is filtral and admits all coproducts. It remains to prove that $\KH/X$ is not well-pointed.

Since $X$ has at least two points, there exist a compact Hausdorff space $Y\neq\emptyset$ and a non-surjective continuous function $f\colon Y\to X$. We claim that $f$, regarded as an object of $\KH/X$, has no points. A point of $f$ would be a continuous function $s\colon X\to Y$ such that $f\circ s$ is the identity of $X$. But such a section $s$ cannot exist because $f$ is not surjective.  Thus, $\KH/X$ admits a non-initial object with no points. It is not difficult to see that this implies that $\KH/X$ is not well-pointed (cf.\ Lemma~\ref{l:positive-has-points} below).
\end{proof}
\begin{remark}
A well-known result in category theory (cf.~e.g.~\cite{Vitale1994} or~\cite[Theorem 4.4.5]{Borceux2}) states that a category $\C$ is equivalent to the category $\Set^T$ of Eilenberg-Moore algebras for a monad $T$ on $\Set$ if, and only if, it satisfies the following conditions:
\begin{enumerate}
\item $\C$ is exact;
\item $\C$ has a regular projective generator $G$;
\item for every set $I$, the copower $\sum_{I}{G}$ exists in $\C$.
\end{enumerate}
If these conditions are satisfied, the monad $T$ on $\Set$ can be defined as the one induced by the adjunction 
\[\sum_{-}{G}\dashv \hom_{\C}(G,-)\colon \C\to\Set.\]
In view of this characterisation of categories monadic over $\Set$, the hypotheses of Theorem~\ref{t:main-characterisation-bis} (along with item $2$ in Lemma~\ref{l:points-ex-mono}) readily imply that $\X$ is equivalent to the category $\Set^T$, where $T$ is the monad induced by the adjunction $\sum_{-}{\1}\dashv \pt\colon \X\to \Set$. 
If we knew that $T$ coincides with the ultrafilter monad on $\Set$, by Manes' characterisation of $\KH$ as the category of algebras for the ultrafilter monad~\cite{Manes1976}, we would conclude that $\X\cong\KH$. In a sense, the difficulty in the proof of our main result resides in the fact that an explicit description of the monad $T$ is not available. 
\end{remark}
%

%%%%%%%%%%%%%%%%%%%%%%%%%%%%%%%%%%%%%%%%%
\section{Richter's Theorem}\label{s:Richter-thm} 
In this section we derive Richter's characterisation of $\KH$~\cite[Remark 4.7]{Richter1992} from Theorem~\ref{t:main-characterisation-bis}. 
\begin{definition}[{\cite[Definition 2.1]{Richter1992}}]
Let $\C$ be a category, and $\1$ an object of $\C$ such that the coproduct $\1+\1$ (denoted by $\2$) exists in $\C$. Further, let $o\in\hom_{\C}(\1,\2)$. For any $X$ in $\C$, a subset $\mathcal{U}\subseteq \hom_{\C}(X,\2)$ is called a \emph{$(\1,o)$-cover} of $X$ provided
\[
\bigcup_{d\in \mathcal{U}}{(d\circ-)^{-1}(o)}=\hom_{\C}(\1,X), \ \text{where} \  d\circ -\colon \hom_{\C}(\1,X)\to \hom_{\C}(\1,\2).
\]
The object $X$ is \emph{$(\1,o)$-compact} if every $(\1,o)$-cover of $X$ admits a finite $(\1,o)$-subcover. (Topological intuition: every cover of $X$ consisting of clopen sets admits a finite subcover.)
\end{definition}
\begin{theorem}[{\cite[Remark 4.7]{Richter1992}}]\label{t:Richter}
Suppose $\C$ is a category admitting an object $\1$ such that the following properties are satisfied:
\begin{enumerate}
\item\label{richter1} $\C$ has all set-indexed copowers of $\1$;
\item\label{richter2} $\1$ is a regular generator in $\C$, i.e.\ for every object $X$ of $\C$, the canonical morphism $\sum_{\hom_{\C}(\1,X)}{\1}\to X$ is a regular epimorphism;
\item\label{richter3} $\C$ admits all coequalisers;
\item\label{richter4} every internal equivalence relation in $\C$ is exact;
\item\label{richter5} \begin{enumerate}
\item $\hom_{\C}(\1,\2)=\{\bot\neq\top\}$;
\item for every set $I$, the co-diagonal morphisms
\[\begin{tikzcd}
\sum_{I}{\1} & \sum_{I}{\1}+\sum_{I}{\1}\cong \sum_{I}{\2} \arrow{l} \arrow{r} & \2
\end{tikzcd}\]
are jointly monic;
\item $\2$ is a coseparator for the full subcategory of $\C$ on the set-indexed copowers of $\1$, i.e.\ for any sets $I,J$ and distinct morphisms $\phi_1,\phi_2\in\hom_{\C}(\sum_{I}{\1},\sum_{J}{\1})$, there is $\psi\in\hom_{\C}(\sum_{J}{\1},\2)$ such that $\psi\circ\phi_1\neq\psi\circ\phi_2$;
\end{enumerate}
\item\label{richter6} there is $o\in\hom_{\C}(\1,\2)$ such that, for every set $I$, $\sum_{I}{\1}$ is $(\1,o)$-compact.
\end{enumerate}
Then $\C$ is equivalent to the category $\KH$ of compact Hausdorff spaces and continuous maps.
\end{theorem}
For the remainder of this section, we assume $\C$ is a category satisfying items~\ref{richter1}--\ref{richter6} in Theorem~\ref{t:Richter}. 
\begin{remark}\label{rm:explanations}
We comment on some of the assumptions in Richter's theorem.
\begin{enumerate}[wide, labelwidth=!, labelindent=10pt]\renewcommand{\labelenumi}{(\roman{enumi})}
\item Item~\ref{richter5}$a$ states that there are exactly two morphisms $\bot,\top\colon\1\to\2$. It follows that the object $\2$ admits exactly one non-trivial automorphism $\sigma\colon \2\to\2$. The image of $\sigma$ under the functor $\hom_{\C}(\1,-)$ is the function $\{\bot,\top\}\to\{\bot,\top\}$ which sends $\bot$ to $\top$, and $\top$ to $\bot$. Further, we will see in Lemma~\ref{l:1-terminal} below that $\1$ is terminal in $\C$, hence $\hom_{\C}(\1,\1)$ consists only of the identity morphism. Thus, item~\ref{richter5}$a$ is equivalent to saying that the functor $\hom_{\C}(\1,-)$ preserves the coproduct $\1+\1$.
\item Item~\ref{richter5}$b$ is a weak form of extensivity of the category $\C$. In fact, we shall see in Lemma~\ref{l:pt-cosquares-of-free} below that it entails that the functor $\hom_{\C}(\1,-)$ preserves the coproducts of the form $\sum_{I}{\1}+\sum_{I}{\1}$, for any set $I$. Hence, the coproduct $\sum_{I}{\1}+\sum_{I}{\1}$ in $\C$ is disjoint.
\item Given the existence of the automorphism $\sigma\colon \2\to\2$ described in item~(i) of this remark, we can assume without loss of generality that $o=\top$ in item~\ref{richter6} of Theorem~\ref{t:Richter}.
\end{enumerate}
\end{remark}
\begin{lemma}[{\cite[p.\ 370]{Richter1992}}]\label{l:1-terminal}
The object $\1$ is terminal in the category $\C$.
\end{lemma}
In view of the previous lemma, hereafter we call \emph{points} the morphisms with domain $\1$ in the category $\C$, and denote the functor $\hom_{\C}(\1,-)$ by $\pt\colon\C\to\Set$. 

The following fact is proved in~\cite[p.\ 371]{Richter1992}. The main ingredient of the proof is an application of item~\ref{richter5}$b$:
\begin{lemma}\label{l:pt-cosquares-of-free}
For every set $I$, $\pt(\sum_{I}{\1}+\sum_{I}{\1})\cong \pt{\sum_{I}{\1}}+\pt{\sum_{I}{\1}}$.
\end{lemma}
The next two lemmas have the following topological interpretations. The first one states that every object of $\C$ is compact (in the appropriate sense), while the second one corresponds to the observation that the union of finitely many clopens is clopen.
\begin{lemma}\label{l:compactness}
Every object of $\C$ is $(\1,\top)$-compact.
\end{lemma}
\begin{proof}
The proof is a straightforward abstraction of the argument showing that the continuous image of a compact space is compact. 
Let $X$ be an arbitrary object of $\C$, and $\mathcal{U}\subseteq \hom_{\C}(X,\2)$ a $(\1,\top)$-cover of $X$.
Write $\epsilon_X\colon \sum_{\pt{X}}{\1}\epi X$ for the regular epimorphism in item~\ref{richter2}, and consider the set
\[
\mathcal{U}'=\{d\circ\epsilon_X\mid d\in \mathcal{U}\}\subseteq \hom_{\C}\big(\sum_{\pt{X}}{\1},\2\big).
\] 
It is not difficult to see that $\mathcal{U}'$ is a $(\1,\top)$-cover of $\sum_{\pt{X}}{\1}$, hence it admits a finite $(\1,\top)$-subcover
\[
\mathcal{U}'_0=\{d_1\circ\epsilon_X,\ldots,d_n\circ\epsilon_X\}
\]
by item~\ref{richter6}. It follows that $\mathcal{U}_0=\{d_1,\ldots,d_n\}$ is a finite $(\1,\top)$-subcover of $\mathcal{U}$.
\end{proof}
\begin{lemma}\label{l:union-of-clopens}
Given morphisms $\phi_1,\ldots,\phi_n\in \hom_{\C}(X,\2)$, there is $\phi\in \hom_{\C}(X,\2)$ such that
\[
(\pt{\phi})^{-1}(\top)=\bigcup_{i=1}^n{(\pt{\phi_i})^{-1}(\top)}.
\]
\end{lemma}
\begin{proof}
We prove the case $n=2$. The statement then follows by a straightforward induction. It suffices to show the existence of a morphism $\cup\in\hom_{\C}(\2\times \2, \2)$ whose image under $\pt\colon\C\to\Set$ is the function 
\[\P(\{\bot,\top\})\to \{\bot,\top\}, \ \ T\mapsto \begin{cases} \bot & \mbox{if} \ T=\emptyset \\ \top & \mbox{otherwise,} \end{cases}\]
for then $\phi=\cup\circ (\phi_1\times \phi_2)\colon X\to \2$ will have the desired property.
To define the morphism $\cup\colon\2\times \2\to \2$, note that $\2+\2\cong\2\times \2$ by Lemma~\ref{l:pt-cosquares-of-free}. Explicitly, an isomorphism is provided by $(\bot\times id_{\2})+(\top\times id_{\2})\colon \2+\2\to \2\times \2$. 
Therefore, up to composition with the latter isomorphism, $\cup\in\hom_{\C}(\2\times \2, \2)$ can be defined as
\[
id_{\2}+(\top\circ {!})\colon \2+\2\to \2, \ \text{where} \ \2\xrightarrow{!}\1\xrightarrow{\top}\2.
\]
It is not difficult to see that $(\pt{\cup})(T)=\bot$ precisely when $T=\emptyset$, thus concluding the proof.
\end{proof}
Item~\ref{richter1} in Theorem~\ref{t:Richter} implies that the functor \[\pt\colon \C\to \Set\] has a left adjoint $\sum_{-}{\1}\colon \Set\to\C$, and item~\ref{richter2} that $\pt$ is of \emph{descent type}. That is, the comparison functor $\C\to\Set^T$ is full and faithful. Here, $T$ is the monad induced by the adjunction $\sum_{-}{\1}\dashv \pt$, and $\Set^T$ is the associated category of Eilenberg-Moore algebras. The category $\C$ admits all coequalisers by item~\ref{richter3}, whence the comparison functor has a left adjoint. In other words, $\C$ can be identified with a reflective subcategory of $\Set^T$. In particular, $\C$ is complete and cocomplete, and the functor $\pt$---being of descent type---is conservative and thus reflects all limits. Hereafter, to improve exposition, we will assume that $\C$ is a replete subcategory of $\Set^T$, i.e.\ $\C$ is closed under isomorphisms in $\Set^T$.

For every object $X$ of $\Set^T$, we denote by $\rho_X\colon X\to Y$ the component at $X$ of the reflection of $\Set^T$ into $\C$.  Employing usual terminology, we say that $\C$ is an \emph{epireflective} (resp.\ \emph{monoreflective}, or \emph{bireflective}) subcategory of $\Set^T$ provided all reflection morphisms $\rho_X$ are epimorphisms (resp.\ monomorphisms, or both monomorphisms and epimorphisms). 
The next lemma shows that $\C$ is closed in $\Set^T$ under taking certain subobjects. It will allow us to deduce in Lemma~\ref{l:bireflective} that $\C$ is a bireflective subcategory of $\Set^T$.
\begin{lemma}\label{l:partial-fullness-on-subobjects}
Let $X$ be an object of $\C$ such that the canonical morphism $X\to \2^{\hom_{\C}(X,\2)}$ is a monomorphism. Then,
for every monomorphism $m\colon S\mono X$ in $\Set^T$, $S$ belongs to $\C$. In particular, for every set $I$ and monomorphism $S\mono \sum_{I}{\1}$ in $\Set^T$, $S$ belongs to $\C$.
\end{lemma}
\begin{proof}
The second part of the statement follows from the first one and the fact that item~\ref{richter5}$c$, along with the completeness of $\C$, entails that the canonical morphism $\sum_{I}{\1}\to \2^{\hom_{\C}(\sum_{I}{\1},\2)}$ is a monomorphism.

Now, let $m\colon S\mono X$ be as in the statement, and write $V=\pt{X}\setminus \pt{S}$. We will construct a morphism \[\xi\colon X\to \2^V\] such that, for each $q\in \pt{X}$, $\xi\circ q\colon \1\to\2^V$ is the constant morphism $\overline{\top}$ of value $\top$ if, and only if, $q$ belongs to $\pt{S}$ (i.e.\ $q$ factors through $m\colon S\mono X$). Note that $\2^V$ belongs to $\C$ because the latter category is closed under limits in $\Set^T$. Because $\C$ is a full subcategory of $\Set^T$ and the functor $\pt\colon\C\to\Set$ reflects limits, it will follow that $S$ is the equaliser in $\C$ of the diagram
\[\begin{tikzcd}
X \arrow[yshift=3pt]{r}{\overline{\top}} \arrow[yshift=-3pt]{r}[swap]{\xi} & \2^V.
\end{tikzcd}\]
In particular, $S$ will belong to $\C$. To construct the morphism $\xi$, we rely on the following fact:
\begin{claim}\label{cl:separation-set-and-point}
Let $X$ be an object of $\C$ such that the canonical morphism $X\to \2^{\hom_{\C}(X,\2)}$ is a monomorphism, and $S$ a subobject of $X$. For every $p\in \pt{X}\setminus \pt{S}$, there exists a morphism $\xi_p\colon X\to \2$ such that $\pt{S}\subseteq (\pt{\xi_p})^{-1}(\top)$ and $p\notin (\pt{\xi_p})^{-1}(\top)$.
\end{claim}
\begin{proof}[of Claim]
Fix $p\in \pt{X}\setminus \pt{S}$. By assumption, for each $p'\in \pt{S}$ there is a morphism $\xi_{p,p'}\colon X\to \2$ such that  
$\xi_{p,p'}\circ m\circ p=\bot\neq\top=\xi_{p,p'}\circ m\circ p'$. The collection
\[
\mathcal{U}=\{\xi_{p,p'}\mid p'\in \pt{S}\}\subseteq \hom_{\C}(X,\2)
\]
satisfies $\pt{S}\subseteq \bigcup_{d\in \mathcal{U}}{(d\circ-)^{-1}(\top)}$ and $p\notin \bigcup_{d\in \mathcal{U}}{(d\circ-)^{-1}(\top)}$.
By Lemma~\ref{l:compactness}, there are finitely many points $p_1',\ldots,p_n'\in\pt{S}$ such that 
\begin{equation*}
\pt{S}\subseteq \bigcup_{i=1}^n{(\pt{\xi_{p,p_i'}})^{-1}(\top)}.
\end{equation*}
In view of Lemma~\ref{l:union-of-clopens}, there exists a morphism $\xi_p\colon X\to \2$ satisfying
\[
(\pt{\xi_p})^{-1}(\top)=\bigcup_{i=1}^n{(\pt{\xi_{p,p_i'}})^{-1}(\top)}.
\]
We have $\pt{S}\subseteq (\pt{\xi_p})^{-1}(\top)$ and $p\notin (\pt{\xi_p})^{-1}(\top)$, as was to be proved.
\end{proof}
Let $\xi\colon X\to \2^V$ be the unique morphism whose composition with the $p$-th projection $\2^V\to \2$ is $\xi_p$, for every $p\in V$. We claim that $\xi$ satisfies the desired property. In one direction, suppose $q\in \pt{\sum_{I}{\1}}$ is such that $q\in(\pt{\xi_p})^{-1}(\top)$ for every $p\in V$. Then $q\notin V$, for otherwise $q\notin (\pt{\xi_q})^{-1}(\top)$. Thus, $q\in V^c=\pt{S}$. Conversely, pick $q\in\pt{S}$. Then, $q\in \pt{S}\subseteq (\pt{\xi_p})^{-1}(\top)$ for every $p\in V$. This concludes the proof of the lemma. 
\end{proof}
\begin{lemma}\label{l:bireflective}
$\C$ is a bireflective subcategory of $\Set^T$. 
\end{lemma}
\begin{proof}
Let $X$ be an object of $\Set^T$, and $I$ a set such that there exists a regular epimorphism $\epsilon\colon \sum_{I}{\1}\epi X$ in $\Set^T$. Denote by
\[\begin{tikzcd}
R \arrow[yshift=3pt]{r}{\alpha_1} \arrow[yshift=-3pt]{r}[swap]{\alpha_2} & \sum_{I}{\1}
\end{tikzcd}\]
the kernel pair of $\epsilon$. In particular, there is a monomorphism $R\mono \sum_{I}{\1}\times\sum_{I}{\1}$ in $\Set^T$. It is not difficult to see that item~\ref{richter5}$c$, along with the completeness of $\C$, entails that the canonical morphism $\sum_{I}{\1}\times\sum_{I}{\1}\to \2^{\hom_{\C}(\sum_{I}{\1}\times\sum_{I}{\1},\2)}$ is a monomorphism.
It follows by Lemma~\ref{l:partial-fullness-on-subobjects} that $R$ belongs to $\C$. Further, $R$ is an internal equivalence relation in $\C$, and therefore it is effective by item~\ref{richter4}. That is, denoting by $\omega\colon \sum_{I}{\1}\to Y$ the coequaliser of $\alpha_1$ and $\alpha_2$ in $\C$, the following is a pullback square in $\C$.
\[\begin{tikzcd}
R \arrow{r}{\alpha_1} \arrow{d}[swap]{\alpha_2} \arrow[dr, phantom, "\lrcorner", very near start] & \sum_{I}{\1} \arrow{d}{\omega} \\
\sum_{I}{\1} \arrow{r}{\omega} & Y
\end{tikzcd}\]
Because the comparison functor $\C\to\Set^T$ preserves limits, the diagram above is also a pullback in $\Set^T$. Thus, the kernel pair of $\epsilon\colon \sum_{I}{\1}\epi X$ coincides with the kernel pair of $\omega\colon \sum_{I}{\1}\to Y$. Notice that $\omega=\rho_X\circ \epsilon$, where $\rho_X\colon X\to Y$ is the reflection morphism. Hence, the kernel pair of $\epsilon$ coincides with the kernel pair of $\rho_X\circ \epsilon$. It follows that $\rho_X$ is injective, i.e.\ a monomorphism in $\Set^T$. Since $X$ is arbitrary, we conclude that $\C$ is monoreflective in $\Set^T$. By~\cite[Proposition 16.3]{Joy1990}, the latter implies that $\C$ is bireflective.
\end{proof}
Our next aim is to show that the category $\Set^T$ is extensive, hence a pretopos. This will entail at once that $\C$ is also a pretopos, because a balanced category has no non-trivial bireflective subcategories. We start by showing that the canonical forgetful functor $U\colon \Set^T\to \Set$ preserves finite coproducts, as this will allow us to lift the extensivity property from $\Set$ to $\Set^T$.
\begin{lemma}\label{l:U-preserves-finite-sums}
The forgetful functor $U\colon \Set^T\to \Set$ preserves finite coproducts.
\end{lemma}
\begin{proof}
An easy adaptation of the proof of~\cite[Lemma 4.1]{Richter1991} shows that $U$ preserves finite coproducts if, and only if, $U(X+X)\cong U(X)+U(X)$ for every $X$ in $\Set^T$. 
Fix $X$ in $\Set^T$, and let $I$ be a set such that there is a regular epimorphism $\epsilon\colon \sum_{I}{\1}\epi X$. 
Write $R\rightrightarrows \sum_{I}{\1}$ for the kernel pair of $\epsilon$. Since $U(\sum_{I}{\1}+\sum_{I}{\1})\cong U(\sum_{I}{\1})+U(\sum_{I}{\1})$ by Lemma~\ref{l:pt-cosquares-of-free}, $U(X+X)$ is isomorphic to the quotient of $U(\sum_{I}{\1})+U(\sum_{I}{\1})$ with respect to the equivalence relation $U(R+R)\subseteq (U(\sum_{I}{\1})+U(\sum_{I}{\1}))^2$. 
We claim that $U(R+R)\cong U(R)+U(R)$, from which it follows that $U(X+X)\cong U(X)+U(X)$.

Note that $R$ and $R+R$ belong to the category $\C$ by Lemma~\ref{l:partial-fullness-on-subobjects}. 
The obvious function $\pt{R}\cup \pt{R}\to \pt(R+R)$ is injective because two points belonging to different components of $R+R$ are separated by the coproduct of the constant morphisms $\overline{\bot}\colon R \xrightarrow{!}\1 \xrightarrow{\bot}\2$ and $\overline{\top}\colon R \xrightarrow{!}\1 \xrightarrow{\top}\2$. Thus, it remains to prove that the function $\pt{R}\cup \pt{R}\to \pt(R+R)$ is surjective. Suppose, by way of contradiction, that there exists $p\in \pt(R+R)\setminus (\pt{R}\cup\pt{R})$. In view of Claim~\ref{cl:separation-set-and-point}, there is a morphism $h\colon R+R\to \2$ such that $(\pt{h})^{-1}(\top)$ contains both copies of $\pt{R}$, but not $p$. In particular, the following diagram commutes.
\[\begin{tikzcd}
R \arrow{r} \arrow{dr}[swap]{\overline{\top}} & R+R \arrow{d}{h} & R \arrow{l} \arrow{dl}{\overline{\top}} \\
{} & \2 &
\end{tikzcd}\]
By the universal property of the coproduct, $h$ is the unique morphism making the diagram above commute. Whence, $h$ coincides with the constant morphism $R+R \xrightarrow{!}\1 \xrightarrow{\top}\2$, contradicting the fact that $h\circ p\neq \top$. This proves $U(R+R)\cong U(R)+U(R)$, thus concluding the proof.
\end{proof}
\begin{proposition}\label{p:C-pretopos}
The category $\Set^T$ is a pretopos, and $\C\cong \Set^T$.
\end{proposition}
\begin{proof}
The forgetful functor $U\colon \Set^T\to \Set$ preserves and reflects limits, and it also preserves (and reflects) finite coproducts by Lemma~\ref{l:U-preserves-finite-sums}. Thus, the extensivity of $\Set$ entails the extensivity of $\Set^T$. This shows that $\Set^T$ is a pretopos, thus a balanced category. By Lemma~\ref{l:bireflective}, we get $\C\cong \Set^T$.
\end{proof}
We can finally show how Richter's theorem follows from the characterisation of the category of compact Hausdorff spaces and continuous maps provided in Theorem~\ref{t:main-characterisation-bis}.
\begin{proof}[of Theorem~\ref{t:Richter}]
Let $\C$ be a category satisfying the properties~\ref{richter1}--\ref{richter6} in the statement of the theorem. Then,  $\C$ is a complete and cocomplete pretopos by Proposition~\ref{p:C-pretopos}. Further, it is non-trivial and well-pointed by items~\ref{richter5}$a$ and~\ref{richter2}, respectively (cf.\ Lemma~\ref{l:1-terminal}). By Theorem~\ref{t:main-characterisation-bis}, in order to conclude that $\C$ is equivalent to $\KH$, it remains to prove that $\C$ is filtral. We claim that, for every set $I$, the copower $\sum_{I}{\1}$ is a filtral object in $\C$. Since $\1$ is a regular generator in $\C$, this will yield the desired conclusion.

Fix an arbitrary set $I$. We must prove that the map
\[
\phi\colon \Sub{\sum_{I}{\1}}\to \F\big(\B\big(\sum_{I}{\1}\big)\big), \ \ S\mapsto \{C\in \B\big(\sum_{I}{\1}\big)\mid S\leq C\}
\]
is a bounded lattice isomorphism, where $\F(\B(\sum_{I}{\1}))$ is the lattice of filters of the Boolean center of $\Sub{\sum_{I}{\1}}$. 
The map $\phi$ is readily seen to preserve arbitrary suprema. Thus, it suffices to show that it is a bijection. 

Let $S\in\Sub{\sum_{I}{\1}}$ be an arbitrary subobject. We prove that $S=\bigwedge{\phi(S)}$. Note that the latter infimum exists because $\C$ is complete, whence mono-complete. It is immediate that $S\leq \bigwedge{\phi(S)}$. For the converse inequality it is enough to show that, whenever $p\notin \pt{S}$, there is $C\in\B(\sum_{I}{\1})$ such that $S\leq C$ and $p\notin C$. By Claim~\ref{cl:separation-set-and-point}, there is a morphism $\xi_p\colon \sum_{I}{\1}\to \2$ such that $\pt{S}\subseteq (\pt{\xi_p})^{-1}(\top)$ and $p\notin (\pt{\xi_p})^{-1}(\top)$. Then, $C=\xi_p^*(\top)$ satisfies the required properties. This shows that $\phi$ is injective. 
For surjectivity, consider a filter $F\in \F(\B(\sum_{I}{\1}))$. We claim that $F=\phi(\bigwedge{F})$. It is clear that $F\subseteq \phi(\bigwedge{F})$, hence it is enough to show that every $D\in \B(\sum_{I}{\1})$ satisfying $\bigwedge{F}\leq D$ belongs to $F$. Fix such a complemented subobject $D$, and write $D^c$ for its complement. Then,
\[
D^c\leq \bigvee{\{C^c\mid C\in F\}}.
\]
Each subobject $C^c$, for $C\in F$, corresponds to a morphism $\xi_{C^c}\colon \sum_{I}{\1}\to\2$ satisfying $\pt{C^c}=(\pt{\xi_{C^c}})^{-1}(\top)$. Applying Lemma~\ref{l:compactness} to the object $\bigvee{\{C^c\mid C\in F\}}$, we find finitely many elements $C_1,\ldots,C_n\in F$ such that
\[
\bigcup_{C\in F}{(\pt{\xi_{C^c}})^{-1}(\top)}=\bigcup_{i=1}^n{(\pt{\xi_{C_i^c}})^{-1}(\top)},
\]
so that $D^c\leq C_1^c\vee\cdots\vee C_n^c$. Therefore, $C_1\wedge\cdots\wedge C_n\leq D$, showing that $D\in F$.
\end{proof}

%%%%%%%%%%%%%%%%%%%%%%%%%%%%%%%%%%%%%%%%%
\section{Decidable objects and Stone spaces}\label{s:decidable}
In this section, we give a characterisation of the category $\BStone$ of Stone spaces and continuous maps in the spirit of Theorem~\ref{t:main-characterisation-bis}. This is Theorem~\ref{t:Stone-charact} below. 
We pointed out in Example~\ref{ex:examples-pretoposes} that $\BStone$ is a positive and coherent category, but it is not exact.  Accordingly, in this section we drop the exactness condition. Before proceeding, recall that any positive and coherent category is extensive. In the following, we will use this fact without further notice.

We start by preparing two lemmas. The first one states that in a well-pointed, positive and coherent category, every non-initial object has a point. The second one shows that in such a category the lattices of subobjects are atomic precisely when the functor of points is conservative.
\begin{lemma}\label{l:positive-has-points}
Let $\X$ be a non-trivial positive and coherent category which is well-pointed. For every non-initial object $X$ in $\X$, the unique morphism $X\to\1$ is a retraction.
\end{lemma}
\begin{proof}
Recall that a monomorphism $m$ is \emph{extremal} if, whenever it is decomposed as $m=f\circ e$ with $e$ epic, then $e$ is an isomorphism. A moment's reflection shows that $(1)$~if $g\circ f$ is an extremal mono, then so is $f$; $(2)$ every extremal mono that is also an epi must be an isomorphism. We claim that the unique morphism $\0\to\1$ in $\X$ is an extremal mono. It is not difficult to see that $\0\to\1$ is an extremal mono if, and only if, for every non-initial object $X$ there is an object $Y$ and two distinct morphisms $X\rightrightarrows Y$. Since $\X$ is positive, we can take $Y=X+X$ along with the coproduct injections $X\rightrightarrows X+X$.

Now, suppose $\pt{X}=\emptyset$. We must prove that $X\cong \0$. Since $\X$ is well-pointed, the canonical morphism $\sum_{\pt{X}}{\1}\to X$ is an epimorphism. But $\sum_{\emptyset}{\1}=\0$, showing that the unique morphism $\0\to X$ is epic. Because the composition $\0\to X\to \1$ is an extremal mono, so is $\0\to X$. Therefore, $\0\to X$ is both an extremal mono and an epimorphism, whence an isomorphism.
\end{proof}
\begin{lemma}\label{l:Spec-is-conservative}
Let $\X$ be a non-trivial positive coherent category which is well-pointed.
The following statements are equivalent:
\begin{enumerate}
\item $\Sub{X}$ is an atomic lattice for every $X$ in $\X$;
\item the functor $\pt\colon \X\to\Set$ is conservative.
\end{enumerate}
\end{lemma}
\begin{proof}
Note that, by Lemmas~\ref{l:points-ex-mono} and~\ref{l:positive-has-points}, $\Sub{X}$ is an atomic lattice for every $X$ in~$\X$ if, and only if, whenever $m$ is a mono in $\X$ such that $\pt{m}$ is a bijection, $m$ is an isomorphism. That is, if and only if $\pt$ is ``conservative on monomorphisms''. We claim that the latter is equivalent to the functor $\pt\colon \X\to\Set$ being conservative \emph{tout court}.
For the non-trivial direction, assume $\pt$ is conservative on monos and consider a morphism $f$ in $\X$ along with its (regular epi, mono) factorisation $m\circ e$. Suppose \[\pt{f}=\pt{m}\circ \pt{e}\] is an isomorphism. We claim that both $e$ and $m$ are isomorphisms. Since $\pt{f}$ is an iso, $\pt{m}$ is an epi. But $\pt{m}$ is also a mono because $\pt$ preserves limits, thus it is a bijection. Since $\pt$ is conservative on monos, $m$ is an iso.
On the other hand, since $\pt{f}$ is an iso, $\pt{e}$ is a mono. The functor $\pt$ being faithful, it reflects monos. We conclude that $e$ is both a mono and a regular epi in $\X$, whence an iso. Therefore, $f$ is an isomorphism.
\end{proof}
Let $\C$ be an extensive category with finite limits. An object $X$ in $\C$ is \emph{decidable} provided the diagonal morphism $\delta_X\colon X\to X\times X$ is complemented, i.e.\ there exists a morphism $\varepsilon_X\colon Y\to X\times X$ in $\C$ such that
\[\begin{tikzcd} X \arrow{r}{\delta_X} & X\times X & Y \arrow{l}[swap]{\varepsilon_X} \end{tikzcd}\]
 is a coproduct diagram. The class of decidable objects contains the initial and terminal objects, and is closed under taking subobjects, finite products, and finite coproducts. For instance, the decidable objects in $\Top$ are the discrete spaces, while in $\KH$ they are the finite discrete spaces. See~\cite{CJ96} for a proof of these statements, and for the basics of the theory of decidable objects. Throughout, we denote by $\dec{\C}$ the full subcategory of $\C$ on the decidable objects. 
\begin{lemma}\label{l:preserves-decidability}
Let $\X$ be a non-trivial positive coherent category which is well-pointed and filtral. The functor $\Spec\colon \X\to \KH$ from~\eqref{eq:functor-Spec-to-KH} is well-defined and preserves decidable objects.
\end{lemma}
\begin{proof}
To see that the functor $\Spec\colon \X\to\KH$ is well-defined, it is enough to verify that $\X$ satisfies the properties in Assumption~\ref{assump:s4}. In turn, this follows from Lemma~\ref{l:positive-has-points}. 

For the second part of the statement, let $X$ be a decidable object in $\X$, and $Y\to X\times X$ the complement of the diagonal of $X$. Since $\Spec$ preserves finite limits by Lemma~\ref{l:spec-pres-limits}, the diagonal of $X\times X$ is mapped to the diagonal of $\Spec{X}\times \Spec{X}$, and admits $\Spec{Y}$ as a complement because $\Spec$ preserves finite coproducts (the same proof as for Claim~\ref{claim:spec-fin-coprod} applies here). Thus, the functor $\Spec$ preserves decidable objects.
\end{proof}
\begin{proposition}\label{p:decidable-finite-sets}
Let $\X$ be a non-trivial positive coherent category which is well-pointed and filtral. Further, assume $\Sub{X}$ is an atomic lattice for every $X$ in $\X$. 
Then the functor $\Spec\colon \X\to \KH$ restricts to an equivalence between $\dec{\X}$ and the category $\Setfin$ of finite sets and functions.
\end{proposition}
\begin{proof}
Since every decidable object in $\KH$ is a finite discrete space, Lemma~\ref{l:preserves-decidability} entails that the functor $\Spec\colon \X\to\KH$ restricts to a functor $\Spec\colon \dec{\X}\to \Setfin$. Because the former is faithful, so is the latter. Fullness follows at once from the following claim:
\begin{claim}\label{claim:fullness-finite}
For every continuous function $f\colon \Spec{X}\to\Spec{Y}$, with $\Spec{Y}$ a finite discrete space, there is a morphism $g\colon X\to Y$ in $\X$ such that $\Spec{g}=f$.
\end{claim}
\begin{proof}[of Claim]
Since $\Spec{Y}$ is a finite discrete space, $f$ induces a partition of $\Spec{X}$ into finitely many clopens. By Lemma~\ref{l:atomic-fixed-point}, these clopens correspond to complemented subobjects $S_1,\ldots,S_n$ of $X$. Thus, $X\cong\sum_{i=1}^n{S_i}$. For each $i\in\{1,\ldots,n\}$, let $p_i\in \pt{Y}$ be the value that $f$ assumes on the clopen corresponding to $S_i$. Define the morphism
$
g_i=p_i \circ {!} \colon S_i\to \1\to Y.
$
Upon writing $g=\sum_{i=1}^n{g_i}\colon X\to Y$, we see that $\Spec{g}=f$.
\end{proof}
It remains to show that the functor $\Spec\colon \dec{\X}\to\Setfin$ is essentially surjective. Suppose $Y$ is a discrete space with $n$ elements. 
Finite coproducts in $\X$ are disjoint and universal, thus every coproduct injection of $\sum_{i=1}^n{\1}$ yields a distinct point of $\sum_{i=1}^n{\1}$, and every point is a coproduct injection.
Hence, $Y\cong \Spec{\sum_{i=1}^n{\1}}$. The object $\sum_{i=1}^n{\1}$ is decidable because it is a finite coproduct of decidable objects. 
\end{proof}
Call \emph{pro-decidable} an object of $\X$ which is the codirected limit of decidable objects, and write $\prodec{\X}$ for the full subcategory of $\X$ on the pro-decidable objects. It turns out that $\prodec{\X}$ is equivalent to the category of Stone spaces and continuous maps:
\begin{proposition}\label{p:prodecidable-stone}
Let $\X$ be a non-trivial positive and coherent category which is well-pointed, filtral and complete. Further, assume that $\Sub{X}$ is an atomic lattice for every $X$ in $\X$. 
The functor $\Spec\colon \X\to \KH$ restricts to an equivalence between $\prodec{\X}$ and $\BStone$.
\end{proposition}
\begin{proof}
The functor $\Spec\colon \X\to\KH$ restricts to a functor $\Spec\colon \prodec{\X}\to \BStone$ by Lemmas~\ref{l:spec-pres-limits} and~\ref{l:preserves-decidability}. Since the former is faithful, so is the latter. Every Stone space is the codirected limit of finite discrete spaces. Further, each finite discrete space is isomorphic to one of the form $\Spec{X}$, for $X$ in $\dec{\X}$, by Proposition~\ref{p:decidable-finite-sets}. Because $\X$ is complete by assumption, and $\Spec$ preserves limits by Lemma~\ref{l:spec-pres-limits}, we deduce that $\Spec\colon \prodec{\X}\to \BStone$ is essentially surjective. To conclude the proof, we must show that it is also full.

Assume $f\colon \Spec{X}\to\Spec{Y}$ is a continuous function and $\Spec{Y}$ is a Stone space. Then, $f$ is uniquely determined by its compositions with the quotients onto the finite discrete images of $\Spec{Y}$. Such finite images are in the essential range of the functor $\Spec\colon \dec{\X}\to\Setfin$, so they are of the form $p_i\colon \Spec{Y}\to\Spec{Y_i}$, with each $Y_i$ decidable. Writing $f_i=p_i\circ f$, the function $f$ is determined by the cone 
\[\{f_i\colon \Spec{X}\to \Spec{Y_i}\mid i\in I\}.\]
By Claim~\ref{claim:fullness-finite}, for each $f_i$ there is $\phi_i\colon X\to Y_i$ such that $\Spec{\phi_i}=f_i$. Similarly, for each $p_i\colon \Spec{Y}\to\Spec{Y_i}$ there is $\pi_i\colon Y\to Y_i$ with $\Spec{\pi_i}=p_i$. The functor $\pt$ is conservative by Lemma~\ref{l:Spec-is-conservative}, hence so is $\Spec$. It follows that $\Spec$ reflects limits. That is, the limit of the codirected system $\{\pi_i\colon Y\to Y_i\mid i\in I\}$ in $\X$ is $Y$. Let $g\colon X\to Y$ be the morphism induced by the cone $\{\phi_i\colon X\to Y_i\mid i\in I\}$ in $\X$. We have \[p_i\circ \Spec{g}=\Spec{(\pi_i\circ g)}=\Spec{\phi_i}=f_i\] for every $i\in I$, whence $\Spec{g}=f$. This concludes the proof.
\end{proof}
Recall our characterisation of the category of compact Hausdorff spaces, up to equivalence, as the unique non-trivial pretopos which is well-pointed, filtral, and admits all set-indexed copowers of its terminal object (Theorem~\ref{t:main-characterisation-bis}). In the same spirit, we obtain the following characterisation of the category of Stone spaces and continuous maps, where by a \emph{strongly filtral} coherent category we understand a coherent category in which every object is filtral.
\begin{theorem}\label{t:Stone-charact}
Up to equivalence, $\BStone$ is the unique non-trivial positive and coherent category which is well-pointed, strongly filtral and complete.
\end{theorem}
\begin{proof}
First, note that $\BStone$ is a non-trivial positive and coherent category which is well-pointed and complete (cf.\ Example~\ref{ex:examples-pretoposes}). Moreover, every object of $\BStone$ is filtral (cf.\ Example~\ref{ex:KH-filtral}).

In the other direction, assume $\X$ is a category satisfying the hypotheses in the statement. For any object $X$ of $\X$, we have $\Sub{X}\cong\F(\B(X))$. Since $\F(\B(X))$ is an atomic lattice, so is $\Sub{X}$. Therefore, Proposition~\ref{p:prodecidable-stone} applies to show that the full subcategory of $\X$ on the pro-decidable objects is equivalent to $\BStone$. We claim that every object $X$ of $\X$ is pro-decidable, from which it will follow that $\X$ is equivalent to $\BStone$.
Since $\X$ is strongly filtral, Lemmas~\ref{l:charact-filtral-object} and~\ref{l:positive-has-points} imply that $\Spec{X}$ is a Stone space, whence the codirected limit in $\KH$ of finite discrete spaces.
The functor $\Spec\colon\X\to\KH$ preserves limits (Lemma~\ref{l:spec-pres-limits}) and is conservative (Lemma~\ref{l:Spec-is-conservative}), thus it reflects limits. Using Claim~\ref{claim:fullness-finite}, we conclude that $X$ is the codirected limit in $\X$ of decidable objects.
\end{proof}
\section{The exact completion of $\BStone$}\label{s:epilogue}
In this last section, we indicate how to exploit Theorem~\ref{t:main-characterisation-bis} to show that the exact completion of $\BStone$ coincides with $\KH$, a folklore result whose proof seems not to have appeared in print. 

If $\C$ is any regular category, the coequaliser of an equivalence relation $R\rightrightarrows X$ in $\C$ need not exist, in general. Even if it exists, call it $e\colon X\epi Y$, it may happen that $e$ ``identifies more points than those prescribed by $R$'' (that is, $R$ does not coincide with the kernel pair of $e$). This means, in a sense, that $\C$ does not have enough quotients to describe all its equivalence relations. The problem of adding the missing quotients to the category $\C$, i.e.\ of completing a regular category to an exact one, has a universal solution. Recall that a functor between regular categories is \emph{exact} if it preserves finite limits and regular epimorphisms. The solution then consists of an exact category $\C_{ex/reg}$---called the \emph{exact}, or \emph{ex/reg completion} of $\C$---along with a fully faithful exact functor 
\[J\colon \C\to\C_{ex/reg}\]
satisfying the following universal property: for every exact category $\D$, precomposition with $J$ yields an equivalence between the category of exact functors $\C\to \D$, and the category of exact functors $\C_{ex/reg}\to\D$. The category $\C_{ex/reg}$ exists and is unique up to equivalence. Its objects can be identified with the equivalence relations in $\C$, and the functor $J\colon \C\to\C_{ex/reg}$ sends an object $X$ in $\C$ to its diagonal, that is the identity relation on $X$.

The exact completion yields a left adjoint to the (full) forgetful 2-functor from the 2-category of exact categories and exact functors, to the 2-category of regular categories and exact functors. In a similar manner, one can define the \emph{pretopos completion} of a coherent category. It corresponds to a left adjoint to the (full) forgetful 2-functor from the 2-category of pretoposes and pretopos morphisms, to the 2-category of coherent categories and coherent functors (cf.\ item $1$ in Proposition~\ref{p:pretopos-morphisms}).
For more background on exact and pretopos completions, we refer the reader to~\cite{CV98} or~\cite[A3.3]{Elephant1} (where $\C_{ex/reg}$ is denoted by $\mathbf{Eff}(\C)$), and~\cite[pp.\ 255-271]{MR77}, respectively.
\begin{theorem}\label{t:exact-completion-stone}
The exact completion of the category $\BStone$ of Stone spaces and continuous maps is the category $\KH$ of compact Hausdorff spaces and continuous maps. Therefore, $\KH$ is also the pretopos completion of $\BStone$.
\end{theorem}
\begin{proof}
Let $\X$ denote the exact completion of $\BStone$, and $J\colon \BStone\to \X$ the full and faithful exact functor satisfying the aforementioned universal property. Because $\BStone$ is extensive, so is $\X$ (cf.~\cite[Lemma 2.2]{Carboni95}). Thus, $\X$ is a non-trivial pretopos. To conclude that $\X$ is equivalent to $\KH$, by Theorem~\ref{t:main-characterisation-bis}, it suffices to show that $\X$ (i)~admits all set-indexed copowers of its terminal object, (ii)~is well-pointed and (iii)~filtral. 
\begin{enumerate}[wide, labelwidth=!, labelindent=10pt]\renewcommand{\labelenumi}{(\roman{enumi})}
\item Since $J$ preserves finite limits, it sends the one-point space $\1$ to the terminal object of $\X$. The latter can be represented as the unique equivalence relation on $\1$, i.e.\ the identity relation $\1\rightrightarrows \1$. A straightforward computation shows that, for any set $I$, the $I$-fold copower of the terminal object in $\X$ coincides with the identity relation on $\sum_{I}{\1}$. Hence, $J$ preserves copowers of the one-point space. In particular, $\X$ admits all set-indexed copowers of the terminal object.
\item Observe that the functor $J$ covers its codomain, i.e.\ for every $Y$ in $\X$ there is a Stone space $X$ and a regular epimorphism $J(X)\epi Y$ (an equivalence relation on an object is covered by the identity relation on the same object).  In turn, $X$ is covered by a copower of the one-point space in $\BStone$. Since $J$ preserves regular epimorphisms and copowers of the one-point space (cf.\ the previous item), $Y$ is covered in $\X$ by a copower of the terminal object. Therefore, $\X$ is well-pointed. 
\item Note that $\Sub{J(X)}\cong \Sub{X}$ for every Stone space $X$ (see, e.g.,~\cite[Lemma 25.21]{McLarty92}, where $\C_{ex/reg}$ is denoted $\mathbf{Map}(\C)$). Thus, by Lemmas~\ref{l:charact-filtral-object} and~\ref{l:positive-has-points}, $J(X)$ is filtral in $\X$. Because $J$ covers its codomain, the category $\X$ is filtral.
\end{enumerate}
This shows that $\KH$ is the exact completion of $\BStone$. 
It is not difficult to see that the functor $J\colon \BStone\to\X$ preserves finite coproducts, whence it is coherent (cf.~\cite[Corollary 3.3.10]{Elephant1}). Since every coherent functor between pretoposes is a pretopos morphism, a fact mentioned in Proposition~\ref{p:pretopos-morphisms}, it follows that $\KH$ is also the pretopos completion of $\BStone$.
\end{proof}

\section*{Acknowledgements}\label{s:acknowledgements}
The second author is grateful to Clemens Berger for several helpful comments on an early draft of this work, and to the participants of the workshop \emph{Logique cat{\'e}gorique, topos et dualit{\'e}s}, held in Nice in January 2018, for their valuable feedback. Further, he is greatly indebted to Mamuka Jibladze, who commented on a preliminary version of the paper and raised the question of how it relates to the folklore result identifying $\KH$ with the exact completion of $\BStone$. Finally, we are grateful to the anonymous referee for their comments on our manuscript.

\refs

\bibitem [Ad\'amek, Herrlich \& Strecker, 1990]{Joy1990}
J.~Ad\'amek, H.~Herrlich, and G.~E.~Strecker, \emph{Abstract and concrete
  categories}, Pure and Applied Mathematics, John Wiley \& Sons, Inc., New
  York, 1990.

\bibitem [Banaschewski, 1984]{Banaschewski1984}
B.~Banaschewski, \emph{More on compact {H}ausdorff spaces and finitary
  duality}, Canad. J. Math. \textbf{36} (1984), no.~6, 1113--1118.
  
\bibitem[Barr, Grillet \& van Osdol, 1971]{BGvO71}
M.~Barr, P.~A.~Grillet, and D.~van Osdol, \emph{Exact categories and categories of sheaves}, Lecture Notes in Mathematics, vol. 236, Springer-Verlag, Berlin, 1971.

\bibitem[Borceux, 1994]{Borceux2}
F.~Borceux, \emph{Handbook of categorical algebra 2, categories and
  structures}, Encyclopedia of mathematics and its applications, Cambridge
  University Press, Cambridge, New York, 1994.

\bibitem[Borceux \& Bourn, 2004]{BB2004}
F.~Borceux and D.~Bourn, \emph{Mal'cev, protomodular, homological and
  semi-abelian categories}, Mathematics and its Applications, vol. 566, Kluwer
  Academic Publishers, Dordrecht, 2004.

\bibitem[Carboni, 1995]{Carboni95}
A.~Carboni, \emph{Some free constructions in realizability and proof theory},
  J. Pure Appl. Algebra \textbf{103} (1995), no.~2, 117--148.

\bibitem[Carboni \& Janelidze, 1996]{CJ96}
A.~Carboni and G.~Janelidze, \emph{Decidable (= separable) objects and
  morphisms in lextensive categories}, J. Pure Appl. Algebra \textbf{110}
  (1996), no.~3, 219--240.

\bibitem[Carboni, Lack \& Walters, 1993]{CLW93}
A.~Carboni, S.~Lack, and R.~F.~C.~Walters, \emph{Introduction to extensive and
  distributive categories}, J. Pure Appl. Algebra \textbf{84} (1993), no.~2,
  145--158.
  
  \bibitem[Carboni \& Vitale, 1998]{CV98}
A.~Carboni and E.~M.~Vitale, \emph{Regular and exact completions}, J. Pure Appl. Algebra \textbf{125}
  (1998), no.~1-3, 79--116.

\bibitem[Duskin, 1969]{Duskin1969}
J.~Duskin, \emph{Variations on {B}eck's tripleability criterion}, Reports of
  the {M}idwest {C}ategory {S}eminar, {III} (S.~Mac~Lane, ed.), Springer,
  Berlin, 1969, pp.~74--129.

\bibitem[Franklin \& Thomas, 1970]{FT1970}
S.~P.~Franklin and B.~V.~S.~Thomas, \emph{A categorical characterization of
  {CH}}, Carnegie-Mellon Univ. Math. Dept. Research Report 70-33 (1970).

\bibitem[Gelfand \& Naimark, 1943]{GN1943}
I.~Gelfand and M.~Naimark, \emph{On the imbedding of normed rings into the ring
  of operators in {H}ilbert space}, Rec. Math. [Mat. Sbornik] N.S.
  \textbf{12(54)} (1943), no.~2, 197--213.

\bibitem[Herrlich \& Strecker, 1971]{HS1971}
H.~Herrlich and G.~E.~Strecker, \emph{Algebra {$\bigcap $}
  topology=compactness}, General Topology and Appl. \textbf{1} (1971),
  283--287.

\bibitem[Isbell, 1982]{Isbell1982}
J.~Isbell, \emph{Generating the algebraic theory of {$C(X)$}}, Algebra
  Universalis \textbf{15} (1982), no.~2, 153--155.

\bibitem[Johnstone, 1986]{Johnstone1986}
P.~T.~Johnstone, \emph{Stone spaces}, Cambridge Studies in Advanced
  Mathematics, vol.~3, Cambridge University Press, 1986, Reprint of the 1982
  edition.

\bibitem[Johnstone, 2002]{Elephant1}
P.~T.~Johnstone, \emph{Sketches of an elephant: a topos theory compendium. {V}ol. 1},
  Oxford Logic Guides, vol.~43, The Clarendon Press, Oxford University Press,
  New York, 2002.

\bibitem[Kakutani, 1941]{Kakutani41}
S.~Kakutani, \emph{Concrete representation of abstract {$(M)$}-spaces. ({A}
  characterization of the space of continuous functions.)}, Ann. of Math.
  \textbf{42} (1941), no.~4, 994--1024.

\bibitem[Krein \& Krein, 1940]{KreinKrein40}
M.~Krein and S.~Krein, \emph{On an inner characteristic of the set of all
  continuous functions defined on a bicompact {H}ausdorff space}, C. R.
  (Doklady) Acad. Sci. URSS (N.S.) \textbf{27} (1940), 427--430.

\bibitem[Lawvere, 1964]{Lawvere1964}
F.~W.~Lawvere, \emph{An elementary theory of the category of sets}, Proc. Nat.
  Acad. Sci. U.S.A. \textbf{52} (1964), 1506--1511.

\bibitem[Lawvere, 2005]{Lawvere2005}
F.~W.~Lawvere, \emph{An elementary theory of the category of sets (long version) with
  commentary}, Repr. Theory Appl. Categ. (2005), no.~11, 1--35, With comments
  by the author and Colin McLarty.

\bibitem[Lawvere, 2006]{Lawvere2006}
F.~W.~Lawvere, \emph{Adjointness in foundations}, Repr. Theory Appl. Categ. (2006),
  no.~16, 1--16, Reprinted from Dialectica {{\bf{23}}} (1969).

\bibitem[Linton, 1966]{Linton66}
F.~E.~J.~Linton, \emph{Some aspects of equational categories}, Proc. {C}onf.
  {C}ategorical {A}lgebra ({L}a {J}olla, {C}alif., 1965) (S.~Eilenberg, D.~K.
  Harrison, H.~R{\"o}hrl, and S.~Mac~Lane, eds.), Springer, New York, 1966,
  pp.~84--94.

\bibitem[Makkai, 1985]{Makkai85}
M.~Makkai, \emph{Ultraproducts and categorical logic}, Methods in mathematical
  logic ({C}aracas, 1983) (C.~A. Di~Prisco, ed.), Lecture Notes in Math., vol.
  1130, Springer, Berlin, 1985, pp.~222--309.

\bibitem[Makkai \& Reyes, 1977]{MR77}
M.~Makkai and G.~E.~Reyes, \emph{First order categorical logic}, Lecture Notes
  in Mathematics, vol. 611, Springer-Verlag, Berlin-New York, 1977.

\bibitem[Manes, 1976]{Manes1976}
E.~G.~Manes, \emph{Algebraic theories}, Springer-Verlag, New York-Heidelberg,
  1976, Graduate Texts in Mathematics, No. 26.

\bibitem[Marra \& Reggio, 2017]{MR2017}
V.~Marra and L.~Reggio, \emph{Stone duality above dimension zero: axiomatising
  the algebraic theory of {${\rm C}(X)$}}, Adv. Math. \textbf{307} (2017),
  253--287.

\bibitem[McLarty, 1992]{McLarty92}
C.~McLarty, \emph{Elementary categories, elementary toposes}, Oxford Logic
  Guides, vol.~21, The Clarendon Press, Oxford University Press, New York,
  1992.

\bibitem[Pedicchio \& Tholen, 2004]{PT2004}
M.~C.~Pedicchio and W.~Tholen (eds.), \emph{Categorical foundations},
  Encyclopedia of Mathematics and its Applications, vol.~97, Cambridge
  University Press, Cambridge, 2004, Special topics in order, topology,
  algebra, and sheaf theory.

\bibitem[Richter, 1991]{Richter1991}
G.~Richter, \emph{Axiomatizing the category of compact {H}ausdorff spaces},
  Category theory at work ({B}remen, 1990) (H.~Herrlich and H.-E. Porst, eds.),
  Res. Exp. Math., vol.~18, Heldermann, Berlin, 1991, pp.~199--215.

\bibitem[Richter, 1992]{Richter1992}
G.~Richter, \emph{Axiomatizing algebraically behaved categories of {H}ausdorff
  spaces}, Category theory 1991 ({M}ontreal, {PQ}, 1991) (R.~A.~G. Seely, ed.),
  CMS Conf. Proc., vol.~13, Amer. Math. Soc., Providence, RI, 1992,
  pp.~367--389.

\bibitem[Richter, 1996]{Richter1996}
G.~Richter, \emph{An elementary approach to ``algebra $\cap$ topology $=$
  compactness''}, Appl. Categ. Structures \textbf{4} (1996), no.~4, 443--446.

\bibitem[Rosick\'y, 1989]{Rosicky1989}
J.~Rosick\'y, \emph{Elementary categories}, Arch. Math. (Basel) \textbf{52}
  (1989), no.~3, 284--288.

\bibitem[Schlomiuk, 1970]{Schlomiuk1970}
D.~I.~Schlomiuk, \emph{An elementary theory of the category of topological
  spaces}, Trans. Amer. Math. Soc. \textbf{149} (1970), 259--278.

\bibitem[Stone, 1936]{Stone1936}
M.~H.~Stone, \emph{The theory of representations for {B}oolean algebras},
  Trans. Amer. Math. Soc. \textbf{40} (1936), no.~1, 37--111.

\bibitem[Thron, 1962]{Thron1962}
W.~J.~Thron, \emph{Lattice-equivalence of topological spaces}, Duke Math. J.
  \textbf{29} (1962), no.~4, 671--679.

\bibitem[Vitale, 1994]{Vitale1994}
E.~M.~Vitale, \emph{On the characterization of monadic categories over
  {$SET$}}, Cahiers Topologie G\'{e}om. Diff\'{e}rentielle Cat\'{e}g.
  \textbf{35} (1994), no.~4, 351--358.

\bibitem[Wattel, 1968]{Wattel1968}
E.~Wattel, \emph{The compactness operator in set theory and topology},
  Mathematical Centre Tracts, vol.~21, Mathematisch Centrum, Amsterdam, 1968.

\bibitem[Yosida, 1941]{Yosida41}
K.~Yosida, \emph{On vector lattice with a unit}, Proc. Imp. Acad. Tokyo
  \textbf{17} (1941), no.~5, 121--124.

\endrefs

\end{document}